\documentclass[review]{siamart190516}   % Use [final] when ready.
\usepackage{bbold}
\usepackage{bm}
\usepackage[caption=false]{subfig}

\usepackage{siampaper}
\usepackage{tikzscale}
\usetikzlibrary{external}
\usetikzlibrary{backgrounds}
\usepackage{pgfplots}
\pgfplotsset{compat=1.16}
\usepgfplotslibrary{fillbetween}

\newcommand*{\includetikzgraphics}[2][]{%
  \includegraphics[#1]{#2}
}

\makeatletter
\renewcommand{\todo}[2][]{\tikzexternaldisable\@todo[#1]{#2}\tikzexternalenable}
\makeatother

\newcommand{\cahiernumber}{12}  % Insert your Cahier du GERAD number.
\title{
  A Proximal Quasi-Newton Trust-Region Method for Nonsmooth Regularized Optimization%
  \thanks{Disclaimer: This report was prepared as an account of work sponsored by an agency of the United
  States Government. Neither the United States Government nor any agency thereof, nor any of their
  employees, makes any warranty, express or implied, or assumes any legal liability or responsibility for the
  accuracy, completeness, or usefulness of any information, apparatus, product, or process disclosed, or
  represents that its use would not infringe privately owned rights. Reference herein to any specific
  commercial product, process, or service by trade name, trademark, manufacturer, or otherwise does not
  necessarily constitute or imply its endorsement, recommendation, or favoring by the United States
  Government or any agency thereof. The views and opinions of authors expressed herein do not
  necessarily state or reflect those of the United States Government or any agency thereof.}
}

\author{
  Aleksandr Y. Aravkin%
  \thanks{%
    Department of Applied Mathematics,
    University of Washington,
    Seattle WA., USA\@.
    E-mail: \mailto{saravkin@uw.edu}.
    Research partially supported by the Washington Research Foundation\@.
  }
  \and
  Robert Baraldi%
  \thanks{%
    Department of Applied Mathematics,
    University of Washington,
    Seattle WA., USA\@.
    E-mail: \mailto{rbaraldi@uw.edu}.
    This material is based upon work supported by the U.S. Department of Energy, Office
    of Science, Office of Advanced Scientific Computing Research, Department of Energy Computational
    Science Graduate Fellowship under Award Number DE-FG02-97ER25308.
  }
  \and
  Dominique Orban%
  \thanks{%
    GERAD and Department of Mathematics and Industrial Engineering,
    Polytechnique Montr\'eal, QC, Canada.
    E-mail: \mailto{dominique.orban@gerad.ca}.
    Research partially supported by an NSERC Discovery Grant.
  }
}
\date{\today}

\newtheorem{problemassumption}{Problem Assumption}[section]
\newtheorem{modelassumption}{Model Assumption}[section]
\newtheorem{stepassumption}{Step Assumption}[section]

\crefname{problemassumption}{Problem Assumption}{Problem Assumptions}
\Crefname{problemassumption}{Problem Assumption}{Problem Assumptions}
\crefname{modelassumption}{Model Assumption}{Model Assumptions}
\Crefname{modelassumption}{Model Assumption}{Model Assumptions}
\crefname{stepassumption}{Step Assumption}{Step Assumptions}
\Crefname{stepassumption}{Step Assumption}{Step Assumptions}

\Crefname{subsection}{Section}{Sections}
\crefname{subsection}{section}{sections}

\newcommand{\proj}[1]{\mathop{\textup{proj}}_{#1}}
\newcommand{\dist}{\mathop{\textup{dist}}}

\begin{document}

  \nolinenumbers
  \maketitle

  \thispagestyle{firstpage}
  \pagestyle{myheadings}

  \begin{abstract}
    We develop a trust-region method for minimizing the sum of a smooth term \(f\) and a nonsmooth term \(h\), both of which can be nonconvex.
    Each iteration of our method minimizes a possibly nonconvex model of \(f + h\) in a trust region.
    The model coincides with \(f + h\) in value and subdifferential at the center.
    We establish global convergence to a first-order stationary point when \(f\) satisfies a smoothness condition that holds, in particular, when it has Lipschitz-continuous gradient, and \(h\) is proper and lower semi-continuous.
    The model of \(h\) is required to be proper, lower-semi-continuous and prox-bounded.
    Under these weak assumptions, we establish a worst-case \(O(1/\epsilon^2)\) iteration complexity bound that matches the best known complexity bound of standard trust-region methods for smooth optimization.
    We detail a special instance, named TR-PG, in which we use a limited-memory quasi-Newton model of \(f\) and compute a step with the proximal gradient method, resulting in a practical proximal quasi-Newton method.
    We establish similar convergence properties and complexity bound for a quadratic regularization variant, named R2, and provide an interpretation as a proximal gradient method with adaptive step size for nonconvex problems.
    R2 may also be used to compute steps inside the trust-region method, resulting in an implementation named TR-R2.
    We describe our Julia implementations and report numerical results on inverse problems from sparse optimization and signal processing.
    Both TR-PG and TR-R2 exhibit promising performance and compare favorably with two linesearch proximal quasi-Newton methods based on convex models.
  \end{abstract}



  %!TEX root = tr-nonsmooth.tex
\section{Introduction}

We consider the problem class
\begin{equation}
  \label{eq:nlp} \minimize{x} \ f(x) + h(x),
\end{equation}
where \(f: \R^n \to \R\) is continuously differentiable, \(h: \R^n \to \R \cup \{+\infty\}\) is proper and lower semi-continuous, and both may be nonconvex.
Smooth and nonsmooth optimization problems are special cases corresponding to \(h := 0\) and \(f := 0\), respectively.
Certain authors \citep{lee-sun-saunders-2014,cartis-gould-toint-2011} refer to~\eqref{eq:nlp} as a \emph{composite} problem.
We use instead the term \emph{nonsmooth regularized} to differentiate with problems where \(f = 0\) and \(h(x) = g(c(x))\), where \(g\) is nonsmooth and \(c\) is smooth, which is indeed the composition of two functions.
In practice, \(h\) is often a regularizer designed to promote desirable properties in solutions, such as sparsity.
The class~\eqref{eq:nlp} captures the natural structure of a wide range of problems; problems with simple constraints, exact penalty formulations, basis selection problems with both convex \citep{tibshirani1996regression,spgl1} and nonconvex \citep{blumensath2009iterative,zheng2018unified,baraldi2019basis} regularization, and more general inverse and learning problems \citep{combettes2011proximal,palm,aravkin2020trimmed}.

We describe a trust-region method for~\eqref{eq:nlp} in which steps are computed by approximately minimizing simpler nonsmooth iteration-dependent models inside a trust region defined by an arbitrary norm.
In practice, the norm is chosen based on the nonsmooth term in the model and the tractability of the step-finding subproblem, which is not required to be convex.
% The smooth term in the model is required to have a Lipschitz-continuous gradient, while the nonsmooth term is required to be proper and lower semi-continuous.
Our analysis hinges on the observation that in the nonsmooth context, the first step of the proximal gradient method is the right generalization of the gradient in smooth optimization.
We establish global convergence in terms of an optimality measure describing the decrease achievable in the model by a single step of the proximal gradient method inside the trust-region.
We also establish a worst-case complexity bound of \(O(1/\epsilon^2)\) iterations to bring this optimality measure below a tolerance \(0 < \epsilon < 1\).
Others \citep{grapiglia-yuan-yuan-2016, cartis-gould-toint-2011} have observed that it is possible to devise trust-region methods for regularized optimization with complexity equivalent to that for smooth optimization.
However, past research typically assumes that \(h\) is either globally Lipschitz continuous and/or convex.
% In constrast, we only require \(f\) to be continuously differentiable, and \(h\) to be proper and lower semi-continuous.

We also revisit a quadratic regularization method, and establish similar convergence properties and same worst-case compexity under the same assumptions.
Our description highlights the connection between the quadratic regularization method and the standard proximal gradient method.
The former may be seen as an implementation of the latter with adaptive step size.

We provide implementation details and illustrate the performance of an instance where the trust-region model is the sum of a limited-memory quasi-Newton, possibly nonconvex, approximation of \(f\) with a nonsmooth model of \(h\) and various choices of the trust-region norm.
Our trust-region algorithm exhibits promising performance and compares favorably with linesearch proximal quasi-Newton methods based on convex models \citep{stella-themelis-sopasakis-patrinos-2017, themelis-stella-patrinos-2017}.
Our open source implementations are available from \https{github.com/UW-AMO/TRNC} as packages in the emerging Julia programing language \citep{bezanson-edelman-karpinski-shah-2017}.

As far as we can tell from the literature, the method described in the present paper is the first trust-region method for the fully nonconvex nonsmooth regularized problem.
Our approach offers flexibility in the choice of the norm used to define the trust-region, provided an efficient procedure is known to solve the subproblem.
We show that such procedures are easily obtained in a number of applied scenarios.

\subsection*{Related research}

We focus on~\eqref{eq:nlp} and do not provide an extensive review of approaches for smooth optimization.
\citet{conn-gould-toint-2000} cover trust-region methods for smooth optimization thoroughly, as well as a number of select generalizations,
and we refer the reader to their comprehensive treatment for background.

\citet{yuan-1985} formulates conditions for convergence of trust-region methods for convex-composite objectives, i.e., \(g(c(x))\) where \(c\) is continuously differentiable and \(g\) is convex.
In particular, he considers models of the form \(s \mapsto g(c(x) + \nabla c(x) s)\), that are relevant to exact penalty methods for constrained optimization, and that are a special case of the models we consider.

\citet{dennis-li-tapia-1995} develop convergence properties of trust-region methods for the case where \(f = 0\) and \(h\) is Lipschitz continuous.
Their analysis is based on a generalization of the concept of Cauchy point in terms of Clarke directional derivatives, but they do not provide an approach to solve the typically nonsmooth subproblem.
\citet{kim-sra-dhillon-2010} analyze a trust-region method for~\eqref{eq:nlp} when \(f\) is convex and \(h\) is continuous and convex with assumptions based on those of \citet{dennis-li-tapia-1995}.
Their model around a current \(x\) has the form \(f(x) + \nabla {f(x)}^T s + \tfrac{1}{2} \alpha \|s\|^2 + h(x + s)\), where \(\alpha\) is a Barzilai-Borwein step length safeguarded to stay sufficiently positive and bounded.
By contrast, our approach allows general quadratic models, possibly indefinite, and explicitly accounts for the trust-region constraint in the subproblem by devising specialized proximal operators.

\citet{qi-sun-1994} propose a trust-region method inspired by that of \citet{dennis-li-tapia-1995} for the case where \(f = 0\) and \(h\) is locally Lipschitz continuous with bounded level sets.
They establish convergence under the further assumption that the models are \([0, \,1]\)-subhomogeneous.
% , i.e., they satisfy \(\psi(\alpha s) \leq \alpha \psi(s)\) for all \(s \in \R^n\) and all \(\alpha \in [0, \, 1]\).
\citet{martinez-moretti-1997} employ similar assumptions to generalize the approach to problems with linear constraints.

\citet{cartis-gould-toint-2011} consider~\eqref{eq:nlp} where \(h\) is convex and globally Lipschitz continuous.
They analyze both a trust-region algorithm and a quadratic regularization variant, develop convergence and iteration complexity results, but do not provide guidance on how to compute steps in practice.
Their analysis revolves around properties of a stationarity measure that are strongly anchored to the convexity assumption.
The algorithms that we develop below are most similar to theirs but rest upon significantly weaker assumptions and concrete subproblem solvers.
\citet{grapiglia-yuan-yuan-2016} detail a unified convergence theory for smooth optimization that has trust-region methods as a special case.
They also generalize the results of \citep{cartis-gould-toint-2011} but focus on objectives of the form \(f(x) + g(c(x))\) where \(f\) and \(c\) are smooth and \(g\) is convex and globally Lipschitz.

\citet{lee-sun-saunders-2014} fully explore the global and fast local convergence properties of exact and inexact proximal Newton and quasi-Newton methods for the case where both \(f\) and \(h\) are convex.
They show that those methods inherit all the desired properties of their counterparts in smooth optimization.

\citet{palm} present a proximal alternating method for objectives of the form \(g(x) + Q(x, y) + h(y)\) where \(g\) and \(h\) are proper and lower semi-continuous and the coupling function \(Q\) is continuously differentiable.
Their setting has~\eqref{eq:nlp} as a special case.
They establish convergence under the Kurdyka-Łojasiewicz assumption and provide a general recipe for algorithmic convergence under such an assumption.

\citet{li-lin-2015} consider monotone and non-monotone accelerations of the proximal gradient method for possibly nonconvex \(f\) and \(h\).
They establish global convergence under the assumptions that \(f\) has a Lipschitz continuous gradient, \(h\) is proper and lower semi-continuous, and that \(f + h\) is coercive.
This leads to a sublinear iteration complexity bound when a Kurdyka-Łojasiewicz condition holds.
\citet{bot-csetnek-laszlo-2016} employ an inertial acceleration strategy which converges under the assumptions that \(h\) is bounded below and possesses a Kurdyka-Łojasiewicz condition.

\citet{stella-themelis-sopasakis-patrinos-2017} initially devised PANOC, a linesearch quasi-Newton method for~\eqref{eq:nlp} with limited-memory BFGS Hessian approximations, for model predictive control.
PANOC assumes that the objective has the form \(f(x) + h_1(x) + h_2(c(x))\), where \(f\) and \(c\) are smooth, \(h_1\) is nonsmooth and may be nonconvex, and \(h_2\) is nonsmooth and convex.
\citet{themelis-stella-patrinos-2017} develop ZeroFPR, a nonmonotone linesearch proximal quasi-Newton method for~\eqref{eq:nlp} based on the concept of forward-backward envelope.
ZeroFPR converges under a Kurdyka-Łojasiewicz assumption and enjoys the fast local convergence properties of quasi-Newton methods for smooth optimization when a Dennis-Mor\'e condition holds.

\subsection*{Notation}

Sets are represented by calligraphic letters.
The cardinality of set \(\mathcal{S}\) is represented by \(|\mathcal{S}|\).
We use \(\|\cdot\|\) to denote a generic norm on \(\R^n\).
The symbols \(\nu\), \(\lambda\), \(\sigma\) and \(\Delta\) are scalars.
\(\mathbb{B}
(0, \Delta)\) is the ball centered at \(0\) with radius \(\Delta > 0\) defined by a norm that should be clear from the context.
We use the shorthands \(\mathbb{B} = \mathbb{B}(0, 1)\) and \(\Delta \mathbb{B} = \mathbb{B}(0, \Delta)\).
When necessary, we write \(\mathbb{B}_p\) to indicate that the \(\ell_p\)-norm is used.
Functional symbols \(f\), \(g\), \(h\), as well as \(\phi\), \(\varphi\) and \(\psi\) are used for functions.
\(\chi(\cdot; A)\) represents the indicator function of \(A \subseteq \R^n\).
In particular, the indicator of \(\mathbb{B}(0, \Delta)\) is denoted \(\chi(\cdot; \Delta \mathbb{B})\) or just \(\chi(\cdot; \Delta)\) when the norm is clear from the context.
We use the alternative notation \(\chi(\cdot; \Delta \mathbb{B}_p)\) to emphasize that the \(\ell_p\)-norm is used to define the ball.
If \(A \subseteq \R^n\) and \(x \in \R^n\), \(\dist(x; A) = \inf \{\|a - x\| \mid a \in A\}\) is the Euclidean distance from \(x\) to \(A\).
If \(A\) is closed and convex, \(\proj{A}(x)\) denotes the unique projection of \(x\) into \(A\), i.e., \(\{\proj{A}(x)\} = \argmin{} \{\|a - x\| \mid a \in A\}\).
Finally, \(j\) and \(k\) are iteration counters.

\subsection*{Roadmap}

The paper proceeds as follows.
In \cref{sec:prelim}, we gather preliminary concepts for trust-region methods and variational analysis used in the theory.
\cref{sec:tr} develops the general trust-region method for~\eqref{eq:nlp}, including the new \Cref{alg:tr-nonsmooth}, and introduces several innovations that yield the main results.
In \cref{sec:prox-qn}, we explain how to compute a trust-region step based on a proximal quasi-Newton model.
New relevant proximal operators needed to implement the trust-region method are studied  in \cref{sec:prox-operators}.
A quadratic regularization variant of the trust-region algorithm together with its convergence analysis are presented in \cref{sec:quad-reg}.
Numerical results and experiments are in \cref{sec:numerical}.
We end with a brief discussion in \cref{sec:conclusion}.

  \section{Preliminaries}%
\label{sec:prelim}
\subsection{Smooth context}

When \(f \in \mathcal{C}^1\) and \(h = 0\) in~\eqref{eq:nlp}, trust-region methods are known for strong convergence properties and favorable numerical performance on both small and large-scale problems.
At an iterate \(x_k\), they compute a step \(s_k\) as an approximate solution of
\begin{equation*}
  \minimize{s} \ m_k(s; x_k) \quad \st \ \|s\| \leq \Delta_k,
\end{equation*}
where \(m_k(\cdot; x_k)\) is a model of \(f\) about \(x_k\), \(\|\cdot\|\) is a norm and \(\Delta_k > 0\) is the trust-region radius.
% \label{eq:trsub}
The predicted decrease \(m_k(0; x_k) - m_k(s_k; x_k)\) is compared to the actual decrease \((f+h)(x_k) - (f+h)(x_k + s_k)\) to decide whether \(s_k\) should be accepted or rejected.
If \(s_k\) is accepted, the iteration is \emph{successful}; otherwise it is \emph{unsuccessful}.
Typically, \(m_k(\cdot; x_k)\) is a quadratic expansion of \(f\) about \(x_k\) and the Euclidean norm is used in the trust region.
The Euclidean norm is favored  because efficient numerical schemes are known for the quadratic subproblem, which can be solved either exactly by way of the method of \citet{more-sorensen-1983} or approximately by way of the truncated conjugate gradient method of \citet{steihaug-1983}.
See \citep{conn-gould-toint-2000} for more information.
% \citet{conn-gould-toint-2000} describe the basic assumptions and convergence properties of trust-region methods in the smooth case in Chapter 6, and common approaches for solving the subproblem in Chapter 7.

\subsection{Nonsmooth context}

We denote \(\overline{\R} = \R \cup \{\pm \infty\}\).
We call \(h: \R^n \to \overline{\R}\) \emph{proper} if \(h(x) > -\infty\) for all \(x\) and \(h(x) < \infty\) for at least one \(x\), and \emph{lower semi-continuous}, or \emph{lsc}, at \(\bar{x}\) if \(\liminf_{x \to \bar{x}} h(x) = h(\bar{x})\).
% Similarly, \(h\) is \emph{upper semi-continuous}, or \emph{usc}, at \(\bar{x}\) if \(\limsup_{x \to \bar{x}} h(x) = h(\bar{x})\).
% If \(h\) is both lsc and usc at \(\bar{x}\), it is continuous at \(\bar{x}\).
We say that \(h\) is \emph{(lower-)level bounded} if all its level sets are bounded.
If \(h\) is proper, lsc and level bounded, then \(\argmin{} h\) is nonempty and compact \citep[Theorem~\(1.9\)]{rtrw}.

\begin{definition}
  For a proper lsc function \(h: \R^n \to \overline{\R}\) and a parameter \(\nu > 0\), the \emph{Moreau envelope} \(e_{\nu h}\) and the \emph{proximal mapping} \(\prox{\nu h}\) are defined by
  \begin{subequations}
    \begin{align}
      \label{eq:def-moreau-env} e_{\nu h}(x) & := \inf_w \tfrac{1}{2}\nu^{-1} \|w - x\|^2 + h(w) = \nu^{-1}\, \inf_w \tfrac{1}{2} \|w - x\|^2 + \nu h(w), \\ \label{eq:def-prox-mapping} \prox{\nu h}(x) & := \argmin{w} \tfrac{1}{2} \nu^{-1} \|w - x\|^2 + h(w) = \argmin{w} \tfrac{1}{2} \|w - x\|^2 + \nu h(w).
    \end{align}
  \end{subequations}
\end{definition}
Under certain assumptions, including strong convexity of the objective of~\eqref{eq:def-prox-mapping}, the set \(\prox{\nu h}(x)\) is a singleton.
However, in general, the set-valued mapping \(\prox{\nu h}\) may be empty or contain multiple elements.
For a given \(h\), the range of parameter values for which the Moreau envelope assumes a finite value is given by the following definition.

\begin{definition}
  The proper lsc function \(h: \R^n \to \overline{\R}\) is \emph{prox-bounded} if there exists \(\nu > 0\) and at least one \(x \in \R^n\) such that \(e_{\nu h}(x) > -\infty\).
  The \emph{threshold of prox-boundedness} \(\nu_h\) of \(h\) is the supremum of all such \(\nu > 0\).
\end{definition}

% We can write \(\nu_h = \min(\infty, \, 1 / \rho_h)\), where \(\rho_h := \inf \{\rho \in \R \mid x \mapsto h(x) + \tfrac{1}{2} \rho \|x\|^2 \}\) is bounded below, with \(\nu_h = \infty\) if \(h\) is bounded below.
If \(h\) is level bounded, then so is \(w \mapsto \tfrac{1}{2} \nu^{-1} \|w - x\|^2 + h(w)\)  for all \(x \in \R^n\) and all \(\nu > 0\), so \(e_{\nu h}(x) > -\infty\) \citep[Theorem~\(1.9\)]{rtrw} and \(h\) is prox-bounded.
The following result summarizes some properties of~\eqref{eq:def-moreau-env}--\eqref{eq:def-prox-mapping}.
Further properties appear in \citep[Theorem~\(1.25\)]{rtrw}.
\begin{proposition}%
  \label{prop:prox}
  Let \(h: \R^n \to \overline{\R}\) be proper lsc and prox-bounded with threshold \(\nu_h > 0\).
  For every \(\nu \in (0, \, \nu_h)\) and all \(x \in \R^n\),
  \begin{enumerate}
    \item \(\prox{\nu h}(x)\) is nonempty and compact;
    \item\label{prop:prox-monotonic} \(e_{\nu h}(x)\) depends continuously on \((\nu, x)\) and \(e_{\nu h}(x) \nearrow h(x)\) as \(\nu \searrow 0\).
    % \item if \(\{x_k\} \to \bar{x}\) and \(\{\nu_k\} \searrow 0\) in such a way that \(\{\|x_k - \bar{x}\| / \nu_k\}\) is bounded, then \(\{e_{\nu_k h}(x_k)\} \to h(\bar{x})\);
    % \item if \(\{x_k\} \to \bar{x}\) and \(\{\nu_k\} \to \bar{\nu} \in (0, \, \nu_h)\) and for each \(k\), \(w_k \in \prox{\nu_k h}(x_k)\), then \(\{w_k\}\) is bounded and all its limit points are in \(\prox{\bar{\nu} h}(\bar{x})\).
  \end{enumerate}
\end{proposition}

\subsection{Optimality conditions}

We use the following notions of subgradient and subdifferential~\citep[Definition~\(8.3\)]{rtrw}.

\begin{definition}[Limiting subdifferential]
  Consider \(\phi: \R^n \rightarrow \overline{\R}\) and \(\bar{x} \in \R^n\) with \(\phi(\bar{x}) < \infty\).
  We say that \(v \in \R^n\) is a \emph{regular subgradient} of \(\phi\) at \(\bar{x}\), and we write \(v \in \hat \partial \phi(\bar{x})\) if
  \[
    \liminf_{x \to \bar{x}} \, \frac{ \phi(x) - \phi(\bar{x}) - v^T (x - \bar{x}) }{ \|x - \bar{x}\| } \geq 0.
    % \phi(x) \geq \phi(\bar{x}) + \langle v, x - \bar{x} \rangle + o(\|x-\bar{x}\|)
  \]
  The set of regular subgradients is also called the \emph{Fr\'echet subdifferential}.
  We say that \(v\) is a \emph{general subgradient} of \(\phi\) at \(\bar{x}\), and we write \(v \in \partial \phi(\bar{x})\), if there are sequences \(\{x_k\}\) and \(\{v_k\}\) such that \(x_k \to \bar{x}\), \(\phi(x_k) \to \phi(\bar{x})\), \(v_k \in \hat \partial \phi(x^k)\) and \(v^k \to v\).
  The set of general subgradients is called the \emph{limiting subdifferential}.
\end{definition}

%The Fr\'echet subdifferential is a subset of the limiting subdifferential.
If \(\phi\) is convex, the Fr\'echet and limiting subdifferentials coincide with the subdifferential of convex analysis.
If \(\phi\) is differentiable at \(x\), \(\partial \phi(x) = \{\nabla \phi(x)\}\) and if \(\phi\) is continuously differentiable at \(x\), \(\hat\partial \phi(x) = \{\nabla \phi(x)\}\) \citep[Section~\(8.8\)]{rtrw}.

In the following, we do not make use of the precise definition of the relevant subdifferential, but merely rely on the following criticality property.

\begin{proposition}[{\protect \citealp[Theorem 10.1]{rtrw}}]%
  \label{prop:rtrw}
  If \(\phi: \R^n \to \overline{\R}\) is proper and has a local minimum at \(\bar{x}\), then \(0 \in \hat\partial \phi(\bar{x}) \subseteq \partial \phi(\bar{x})\).
  If \(\phi\) is convex, the latter condition is also sufficient for \(\bar{x}\) to be a global minimum.
  If \(\phi = f + h\) where \(f\) is continuously differentiable on a neighborhood of \(\bar{x}\) and \(h\) is finite at \(\bar{x}\), then  \(\partial \phi(\bar{x}) = \nabla f(\bar{x}) + \partial h(\bar{x})\).
\end{proposition}

\subsection{The proximal gradient method}%
\label{sec:prox_grad_method}

% Before describing the main trust-region algorithm, we first review the proximal-gradient method in our context. % specific problem formulation.
Consider the generic nonsmooth regularized problem
\begin{equation}
  \label{eq:subproblem}
  \minimize{s} \ \varphi(s) + \psi(s),
\end{equation}
where \(\varphi\) is continuously differentiable and \(\psi\) is proper, lower semi-continuous and prox-bounded.
The notation \(\varphi\) and \(\psi\) is intentionally different from~\eqref{eq:nlp} and will be reused to denote models of \(f\) and \(h\) in \cref{sec:tr}.
% We do not emphasize the possible presence of an indicator function in~\eqref{eq:subproblem}, as we assume that it is bundled with \(\psi\).
% To implement the method, we develop new proximity operators for the resulting composite \(\psi\) in Section~\ref{sec:prox-operators}.

A natural method to solve~\eqref{eq:subproblem} that generalizes the gradient method of smooth optimization is the \emph{proximal gradient method} \citep{lions1979fbs,combettes2017operators}.
When initialized from \(s_0 \in \R^n\) where \(\psi\) is finite, it generates iterates according to
\begin{equation}%
  \label{eq:pg}
  s_{j+1} \in \prox{\nu \psi}(s_j - \nu \nabla \varphi(s_j)), \quad j \geq 0,
\end{equation}
where \(\nu > 0\) is a step size.
If \(\psi\) is the indicator of a closed convex set, the proximal gradient method reduces to the projected gradient method.

The first-order optimality conditions of~\eqref{eq:pg} are
\begin{equation}
  \label{eq:prox-kkt}
  0 \in s_{j+1} - s_j + \nu \nabla \varphi(s_j) + \nu \partial\psi(s_{j+1}). 
\end{equation}
The proximal literature primarily focuses on the generalized gradient
\begin{equation}
  \label{eq:moreau_grad}
  G_{\nu}(s) := \nu^{-1} (s - \prox{\nu \psi}(s - \nu \nabla \varphi(s))),
\end{equation}
with \(G_{\nu}(0) = \nabla \varphi(0)\) in the case of smooth optimization.
The following result gives conditions under which the proximal gradient method is monotonic.

\begin{proposition}[{\protect \citealp[Lemma~\(2\)]{palm}}]%
  \label{prop:pg-descent}
  Let \(\varphi\) be continuously differentiable, \(\nabla \varphi\) be Lipschitz continuous with constant \(L > 0\) and \(\psi\) be proper, lsc and bounded below.
  For any \(0 < \nu < 1/L\), any \(s_0\) where \(\psi\) is finite, the iteration~\eqref{eq:pg} is such that
  \[
    (\varphi + \psi)(s_{j+1}) \leq (\varphi + \psi)(s_j) - \tfrac{1}{2} (\nu^{-1} - L) \|s_{j+1} - s_j\|^2, \quad j \geq 0.
  \]
\end{proposition}

% Stronger assertions exist when \(\psi\) is convex, but we focus on general \(\psi\). 
It is possible to remove the assumption that \(\psi\) is bounded below from \Cref{prop:pg-descent} and replace it with the weaker assumption that \(\psi\) is prox-bounded and that \(\nu\) is chosen smaller than the threshold of prox-boundedness of \(\psi\).

In the smooth case, where \(\psi = 0\), we have \(s_1 = -\nu \nabla \varphi(s_0)\) and the decrease is
\begin{equation}
  \label{eq:decrease-smooth-palm}
  \varphi(s_1; x) \leq \varphi(s_0) - \tfrac{1}{2} \nu^2 (\nu^{-1} - L) \|\nabla \varphi(s_0)\|^2.
\end{equation}

  \section{Trust-region methods for nonsmooth regularized optimization}%
\label{sec:tr}

In this section, we develop and analyze a general trust-region method for~\eqref{eq:nlp}.
Section~\ref{sec:tr-properties} examines properties of trust-region subproblems. 
Section~\ref{sec:optim-measures} discusses optimality measures, and highlights the role of the prox-gradient step  in quantifying descent in the general context of~\eqref{eq:nlp}. 
In Section~\ref{sec:trust_region_general}, we present the trust-region approach, and highlight key innovations that make it possible to obtain the convergence results and complexity analysis presented in Section~\ref{sec:tr_analysis}.

\subsection{Properties of trust-region subproblems}%
\label{sec:tr-properties}
For fixed \(x \in \R^n\), consider the parametric problem and its optimal set
\begin{subequations}%
  \label{eq:tr-parametric}
  % \hypertarget{tgt:tr-parametric}{}\label{eq:tr-parametric}
  % \tag*{\(\mathcal{P}(\Delta; x)\)}
  \begin{align}
    \label{eq:tr-parametric-val}
    p(\Delta; x) & := \minimize{s} \ \varphi(s; x) + \psi(s; x) + \chi(s; \Delta), \\
    \label{eq:tr-parametric-opt}
    P(\Delta; x) & := \argmin{s} \ \varphi(s; x) + \psi(s; x) + \chi(s; \Delta),
  \end{align}
\end{subequations}
where \(\varphi(s; x) \approx f(x + s)\), \(\psi(s; x) \approx h(x + s)\), \(\chi(s; \Delta)\) is the indicator function of the trust region \(\Delta \mathbb{B}\) and \(\Delta > 0\).
The form of~\eqref{eq:tr-parametric} is representative of a trust-region subproblem for~\eqref{eq:nlp} in which \(f\) and \(h\) are modeled separately and the trust-region constraint appears implicitly via an indicator function.

We make the following additional assumption.
\begin{modelassumption}%
  \label{asm:parametric}
  For any \(x \in \R^n\), \(\varphi(\cdot; x)\) is continuously differentiable, \(\psi(\cdot; x)\) is proper and lsc.
  % , and \(\varphi(\cdot; x) + \psi(\cdot; x) + \chi(\cdot; \Delta)\) is level-bounded for any \(\Delta \geq 0\).
\end{modelassumption}

By \Cref{prop:rtrw},
\[
  s \in P(\Delta; x) \quad \Longrightarrow \quad 0 \in  \nabla \varphi(s; x) + \partial (\psi(\cdot; x) + \chi(\cdot; \Delta))(s).
\]

The following result summarizes properties of~\eqref{eq:tr-parametric}.

\begin{proposition}%
  \label{prop:tr-parametric}
  Let \Cref{asm:parametric} be satisfied.
  If we define \(p(0; x) := \varphi(0; x) + \psi(0; x)\) and \(P(0; x) = \{0\}\), the domain of \(p(\cdot; x)\) and \(P(\cdot; x)\) is \(\{\Delta \mid \Delta \geq 0\}\).
  In addition,
  \begin{enumerate}
    % \item\label{itm:p-cont} for any \(\Delta \geq 0\), \(p(\Delta; \cdot)\) is continuous if either
    %   \begin{enumerate}
    %     \item \(\varphi(\cdot; x)\) is quadratic and strictly convex; or
    %     \item \(\varphi(\cdot; x) + \psi(\cdot; x)\) is convex;
    %     \item \(\psi(\cdot; x)\) is continuous;
    %   \end{enumerate}
    \item\label{itm:p-lsc} \(p(\cdot; x)\) is proper lsc and for each \(\Delta \geq 0\), \(P(\Delta; x)\) is nonempty and compact;
    \item\label{itm:P-osc} if \(\{\Delta_k\} \to \bar{\Delta} \geq 0\) in such a way that \(\{p(\Delta_k; x)\} \to p(\bar{\Delta}; x)\), and for each \(k\), \(s_k \in P(\Delta_k; x)\), then \(\{s_k\}\) is bounded and all its limit points are in \(P(\bar{\Delta}; x)\);
    \item\label{itm:P-single-valued} if \(\varphi(\cdot; x) + \psi(\cdot; x)\) is strictly convex, \(P(\Delta; x)\) is single-valued;
    \item\label{itm:osc-sufficient} if \(\bar{\Delta} > 0\) and there exists \(\bar{s} \in P(\bar{\Delta}; x)\) such that \(\|\bar{s}\| < \bar{\Delta}\), then \(p(\cdot; x)\) is continuous at \(\bar{\Delta}\) and \(\{p(\Delta_k; x)\} \to p(\bar{\Delta}; x)\) holds in part~\ref{itm:P-osc}.
  \end{enumerate}
\end{proposition}

\begin{proof}
  % Part~\ref{itm:p-cont} follows similarly to \Cref{prop:prox} as in the proof of \citep[Theorem~\(1.25\)]{rtrw} when \(\varphi(\cdot; x)\) is quadratic and strongly convex, and from \citep[Exercice~\(7.45\)]{rtrw} in the other cases.
  \Cref{asm:parametric} and compactness of the trust region ensure that the objective of~\eqref{eq:tr-parametric-val} is always level-bounded in \(s\) locally uniformly in \(\Delta\) \citep[Definition~\(1.16\)]{rtrw} because for any \(\bar{\Delta} > 0\) and \(\epsilon > 0\), and for any \(\Delta \in (\bar{\Delta} - \epsilon, \, \bar{\Delta} + \epsilon)\) with \(\Delta \geq 0\), the level sets of \(\varphi(\cdot; x) + \psi(\cdot; x) + \chi(\cdot; \Delta)\) are contained in \(\Delta \mathbb{B} \subseteq (\bar{\Delta} + \epsilon) \mathbb{B}\).
  Parts~\ref{itm:p-lsc}--\ref{itm:P-osc} follow by \citet{rtrw}, Theorems 1.17 and 7.41. 
  Part~\ref{itm:P-single-valued} follows from \citet[Exercice~\(7.45\)]{rtrw}.
  Part~\ref{itm:osc-sufficient} follows by noting that if \(\|\bar{s}\| < \bar{\Delta}\), then \(\varphi(\bar{s}; x) + \psi(\bar{s}; x) + \chi(\bar{s}; \Delta)\) is continuous in \(\Delta\) in a neighborhood of \(\bar{\Delta}\); the rest follows from \citet[Theorem~\(1.17c\)]{rtrw}.
\end{proof}

% commented this out for now, since we don't come back to outer semi-continuity again. 
%\Cref{prop:tr-parametric} part~\ref{itm:P-osc} states that the mapping \(P(\cdot; x)\) is outer semi-continuous \citep[Definition~\(5.4\)]{rtrw} with respect to \(p(\cdot; x)\)-attentive convergence, i.e.,
%\[
%  \limsup_{\Delta_k \to_p \bar{\Delta}} P(\Delta_k; x) \subseteq P(\bar{\Delta}; x),
%\]
%where \(\to_p\) means that \(\{\Delta_k\} \to \bar{\Delta}\) in such a way that \(\{p(\Delta_k; x)\} \to p(\bar{\Delta}; x)\).

It is not necessary to assume that \(\psi(\cdot; x)\) is prox-bounded in \Cref{asm:parametric} because under the assumptions stated and compactness of the trust region, the objective of~\eqref{eq:tr-parametric-val} is necessarily bounded below, and therefore prox-bounded.
\Cref{prop:tr-parametric} allows us to think of how approximate solutions ``truncated'' by a trust-region constraint approach \(\bar{s}\) as the trust-region radius increases.
Indeed, we may choose any \(\bar{\Delta} > \|\bar{s}\|\) in parts~\ref{itm:P-osc} and~\ref{itm:osc-sufficient}.
When \(\psi(\cdot; x) = 0\) and \(\varphi(\cdot; x)\) is quadratic and strictly convex, the graph of \(P(\cdot; x)\) is known to be a smooth curve such that \(P(0; x) = \{x_k\}\), that is tangential to \(-\nabla f(x_k)\) at \(\Delta = 0\) and such that \(\lim_{\Delta \to \infty} P(\Delta; x)\) contains the Newton step as its only element.
This observation gives rise to several numerical methods to approximate the solution of~\eqref{eq:tr-parametric}, including the dogleg \citep{powell-1970} and double dogleg methods \citep{dennis-mei-1979}.

\subsection{Optimality measures}%
\label{sec:optim-measures}

In this section, we seek a convenient way of assessing whether a given \(x\) is first-order critical for~\eqref{eq:nlp} based on the trust-region subproblem~\eqref{eq:tr-parametric}.
We begin with the following result.

\begin{proposition}%
  \label{prop:g-delta}
  Let \Cref{asm:parametric} be satisfied.
  Assume in addition that \(\nabla_s \varphi(0; x) = \nabla f(x)\), \(\partial \psi(0; x) = \partial h(x)\), and let \(\Delta > 0\).
  Then \(0 \in P(\Delta; x) \Longrightarrow s = 0\) is first-order stationary for~\eqref{eq:tr-parametric} \(\Longleftrightarrow x\) is first-order stationary for~\eqref{eq:nlp}.
\end{proposition}

\begin{proof}
  By definition, \(x\) is first-order stationary if and only if \(0 \in \nabla f(x) + \partial h(x) = \nabla_s \varphi(0; x) + \partial \psi(0; x)\).
  But \(\psi(0; x) = \psi(0; x) + \chi(0; \Delta)\) and \(\partial (\psi(\cdot; x) + \chi(\cdot; \Delta))(0) = \partial \psi(0; x) + \partial \chi(0; \Delta)\) because \(\partial \chi(0; \Delta) = \{0\}\).
  Thus we obtain \(0 \in \nabla_s \varphi(0; x) + \partial (\psi(\cdot; x) + \chi(\cdot; \Delta))(0)\), i.e., \(s = 0\) is first-order stationary for~\eqref{eq:tr-parametric}.
\end{proof}

\Cref{prop:g-delta} suggests we may use an element of \(P(\Delta; x)\) as first-order optimality measure for any \(\Delta > 0\), such as for example \(\|g(\Delta; x)\|\), where \(g(\Delta; x)\) is the least-norm element of \(P(\Delta; x)\).
However, the dependency on \(\Delta\) is inconvenient.
% As a workaround, we might choose instead \(\|g(1; x)\|\) as in \citet{cartis-gould-toint-2011}, but the difficulty of computing \(g(1; x)\) remains.\smarttodo{If we don't use it, do we have to state it?}
In order to circumvent this difficulty, we focus our attention temporarily on the choice
\begin{equation}
  \label{eq:varphi-pg}
  \begin{aligned}
    \varphi(s; x) & = f(x) + \nabla {f(x)}^T s + \tfrac{1}{2} \nu^{-1} \|s\|^2
    \\ & = \tfrac{1}{2} \nu^{-1} \|s + \nu \nabla f(x)\|^2 + f(x) - \tfrac{1}{2} \nu \|\nabla f(x)\|^2,
  \end{aligned}
\end{equation}
where \(\nu > 0\) is fixed, so that for any \(x \in \R^n\),
\begin{subequations}%
  \label{eq:prox-delta}
  \begin{align}
    \label{eq:prox-delta-val}
    p(\Delta; x, \nu) & = e_{\nu \psi(\cdot; x) + \chi(\cdot; \Delta)} (-\nu \nabla f(x)) + f(x) - \tfrac{1}{2} \nu \|\nabla f(x)\|^2,
    % = \inf_s \ q^I(s; x) + \psi(s; x) + \chi(s; \Delta),
    \\
    \label{eq:Prox-delta-opt}
    P(\Delta; x, \nu) & = \prox{\nu \psi(\cdot; x) + \chi(\cdot; \Delta)} (-\nu \nabla f(x)),
    % = \argmin{s} \ q^I(s; x) + \psi(s; x) + \chi(s; \Delta).
  \end{align}
\end{subequations}
and \(p\) only differs from a Moreau envelope by a constant.
The above choice of \(\varphi(\cdot; x)\) allows us to derive a convenient, computable optimality measure, and to generalize the concept of decrease along the steepest descent direction, also known as Cauchy decrease, which is so fundamental to the convergence analysis of computational methods for smooth optimization.

In the special case where \(\psi(\cdot; x) = 0\), \Cref{prop:tr-parametric} part~\ref{itm:P-single-valued} indicates that \(P(\Delta; x, \nu)\) is single valued,
and its only element is the projection of \(-\nu \nabla f(x)\) into the trust region.
On the other hand, \(p(\Delta; x, \nu)\) measures the decrease of~\eqref{eq:varphi-pg} in the direction of the projected gradient.
\citet{cartis-gould-toint-2011} study the special case where \(h(x) = g(c(x))\) with \(g\) convex and globally Lipschitz continuous, and \(c\) smooth.
In lieu of~\eqref{eq:prox-delta-val}, they minimize \(f(x) + \nabla f{(x)}^T s + g(c(x) + \nabla c{(x)}^T s)\) in the trust region, which is analogous.
% Specifically, if \(\psi(s; x) = g(c(x) + \nabla c{(x)}^T s)\), then \(0 \in \prox{\nu \psi(\cdot; x)}(-\nu \nabla f(x))\) if and only if \(0 \in \prox{h}(-\nu \nabla f(x))\), so that~\eqref{eq:prox-delta} is more general.

Crucially,~\eqref{eq:prox-delta} describes the first step of the proximal gradient method with step size \(\nu\) applied to~\eqref{eq:tr-parametric-val} where \(\varphi(\cdot; x)\) is as in~\eqref{eq:varphi-pg} from \(s = 0\) with a trust region of radius \(\Delta\).
In the notation of \cref{sec:prox_grad_method}, \(\varphi\) is \(\varphi(\cdot; x)\) and \(\psi\) is \(\psi(\cdot; x) + \chi(\cdot; \Delta)\).
If \(\psi(\cdot; x)\) is finite at \(s_0 = 0\), the first step of the proximal gradient method is
\begin{equation}%
  \label{eq:pg1}
  \begin{aligned}%
  s_1 & \in \argmin{s} \ \tfrac{1}{2} \nu^{-1} \|s + \nu \nabla f(x)\|^2 + \psi(s; x) + \chi(s; \Delta)
  \\ & = \argmin{s} \ f(x) + \nabla {f(x)}^T s + \tfrac{1}{2} \nu^{-1} \|s\|^2 + \psi(s; x) + \chi(s; \Delta),
  \end{aligned}
\end{equation}
and yields the decrease
\begin{equation}%
  (\varphi + \psi)(s_1; x) \leq (f + h)(x) - \tfrac{1}{2} (\nu^{-1} - L) \|s_1\|^2\label{eq:pg1-decrease-palm}
\end{equation}
Moreover, \(s_1\) is also the first step of the proximal-gradient method applied to~\eqref{eq:tr-parametric-val} where \(\varphi(\cdot; x)\) is \emph{any} model of \(f\) about \(x\) that is differentiable at \(s = 0\) with \(\nabla_s \varphi(0; x) = \nabla f(x)\), and, in particular, any quadratic expansion of \(f\) about \(x\).
In the sequel, we use \(s_1\) as the appropriate generalization to the nonsmooth context of the projected gradient step, which allows us to derive an adequate optimality measure.

% If \(\psi(\cdot; x)\) is proper, lsc and level-bounded, \Cref{prop:tr-parametric} ensures that \(P(\Delta; x)\) in~\eqref{eq:Prox-delta-opt} is closed and compact for any \(\Delta \geq 0\).

Let
\begin{equation}
  \label{eq:optim-measure}
  \xi(\Delta; x, \nu) := f(x) + h(x) - p(\Delta; x, \nu),
\end{equation}
where \(p(\Delta; x, \nu)\) is defined in~\eqref{eq:prox-delta-val}.
In view of the above, \(\xi(\Delta; x, \nu)\) measures the decrease predicted by the first step of the proximal gradient method applied to~\eqref{eq:tr-parametric-val} from \(s = 0\) with trust-region radius \(\Delta\) and step length \(\nu > 0\), where \(\varphi(\cdot; x)\) is any model of \(f\) about \(x\) that is differentiable at \(s = 0\) with \(\nabla_s \varphi(0; x) = \nabla f(x)\).

Assume from now on that \(\varphi(0; x) = f(x)\) and \(\psi(0; x) = h(x)\).
Because \(p(\Delta; x, \nu) \leq \varphi(0; x) + \psi(0; x) + \chi(0; \Delta) = f(x) + h(x)\), we necessarily have \(\xi(\Delta; x, \nu) \geq 0\).

Examples of models of \(f\) satisfying the above assumptions include Taylor expansions of \(f\) about \(x\), and in particular quadratic models \(f(x) + \nabla {f(x)}^T s + \tfrac{1}{2} s^T B s\) where \(B = B^T\).
The most straightforward example of a model of \(h\) satisfying the above is \(\psi(s; x) = h(x + s)\).
If \(h(x) = g(c(x))\), where \(g: \R^m \to \overline{R}\) is proper, lsc and level-bounded, and \(c: \R^n \to \R^m\) is continuously differentiable, other possible models include \(\psi(s; x) = g(c(x) + \nabla {c(x)}^T s)\) and \(\psi(s; x) = g(c(x) + \nabla {c(x)}^T s + \sum_{i=1}^m s^T B_i s)\), where each \(B_i = B_i^T\).

The following result allows us to rely on the computable values \(p(\Delta; x, \nu)\) and \(\xi(\Delta; x, \nu)\) to assess stationarity.

\begin{proposition}%
  \label{prop:xi-critical}
  Let \Cref{asm:parametric} be satisfied where \(\varphi(0; x) = f(x)\) and \(\nabla_s \varphi(0; x) = \nabla f(x)\).
  Assume furthermore that \(\psi(0; x) = h(x)\) and \(\partial \psi(0; x) = \partial h(x)\), and let \(\Delta > 0\).
  Then, \(\xi(\Delta; x, \nu) = 0 \Longleftrightarrow 0 \in P(\Delta; x, \nu) \Longrightarrow x\) is first-order stationary for~\eqref{eq:nlp}.
\end{proposition}

\begin{proof}
  \(\xi(\Delta; x, \nu) = 0\) if and only if \(p(\Delta; x, \nu) = f(x) + h(x) = \varphi(0; x) + \psi(0; x) + \chi(0; \Delta)\), which occurs if and only if \(0 \in P(\Delta; x, \nu)\).
  \Cref{prop:g-delta} then implies that \(x\) is first-order stationary for~\eqref{eq:nlp}.
  % By \Cref{prop:g-delta}, substituting \(s = 0\) in the objective of~\eqref{eq:prox-delta-val} implies \(\xi(\Delta; x) = 0\).
  % % Conversely, \Cref{cor:step_bound} establishes that \(\xi(\Delta; x) \geq \tfrac{1}{2} \theta \nu \|G_\nu(0)\|^2\) so that if \(\xi(\Delta; x) = 0\), we also have \(G_\nu(0) = 0\) and \Cref{lem:s1min=0} implies that \(x\) is stationary.
  % Conversely, if \(\xi(\Delta; x) = 0\), then \(p(\Delta; x) = f(x) + h(x)\), which is the value of \(\varphi(s; x) + \psi(s; x)\) at \(s = 0\) and means that \(s = 0\) is first-order stationary for~\eqref{eq:tr-parametric}.
\end{proof}

\subsection{A trust-region algorithm}%
\label{sec:trust_region_general}

We focus on the solution of~\eqref{eq:nlp} under \Cref{asm:prob}.

\begin{problemassumption}%
  \label{asm:prob}
  In~\eqref{eq:nlp}, \(f \in \mathcal{C}^1(\R^n)\), and \(h\) is proper and lsc.
\end{problemassumption}

At iteration \(k\), we construct a model \(m_k(s; x_k) := \varphi(s; x_k) + \psi(s; x_k) \approx f(x_k + s) + h(x_k + s)\) and we approximately solve
\begin{equation}
  \label{eq:trsub}
  \minimize{s} \ m_k(s; x_k) \quad \st \ \|s\| \leq \Delta_k
\end{equation}
by computing a step \(s_k\) required to result in at least a fraction of the decrease achieved with one step of the proximal gradient method. % applied with steplength \(\nu_k\).
\Cref{asm:step_nonlip_xi} formalizes our requirement.

\begin{stepassumption}%
  \label{asm:step_nonlip_xi}
  There exists \(\kappa_{\textup{m}} > 0\) and \(\kappa_{\textup{mdc}} \in (0, \, 1)\) such that for all \(k\), \(\|s_k\| \leq \Delta_k\) and
  \begin{subequations}
    \begin{align}
      \label{eq:model-adequation-nonlipschitz}
      |f(x_k + s_k) + h(x_k + s_k) - m_k(s_k; x_k)| & \leq \kappa_{\textup{m}} \|s_k\|^2,
      \\
      \label{eq:decrease_nonlipschitz_xi}
      m_k(0; x_k) - m_k(s_k; x_k) & \geq \kappa_{\textup{mdc}} \xi(\Delta_k; x_k, \nu_k),
    \end{align}
  \end{subequations}
  where \(m_k\) is defined above and \(\xi(\Delta_k; x_k, \nu_k)\) is defined in~\eqref{eq:optim-measure}.
\end{stepassumption}

Condition~\eqref{eq:model-adequation-nonlipschitz} is certainly satisfied if both \(f\) and \(\varphi\) are twice continuously differentiable with bounded second derivatives, and \(\psi(s; x_k) := h(x_k + s)\).
It also holds when \(h(x) = g(c(x))\) where \(c: \R^n \to \R^m\) has Lipschitz-continuous Jacobian and \(g: \R^m \to \R^n\) is Lipschitz continuous.
Such a situation arises when~\eqref{eq:nlp} results from penalizing infeasibility in the process of solving a smooth constrained problem.
A useful model is then \(\psi(s; x_k) := g(c(x_k) + \nabla {c(x_k)}^T s)\).
If \(L > 0\) is the Lipschitz constant of \(g\) and \(M > 0\) that of the Jacobian of \(c\), we have
\[
  |h(x_k + s) - \psi(s; x_k)| \leq
  L \|c(x_k + s) - c(x_k) - \nabla {c(x_k)}^T s\| \leq
  \tfrac{1}{2} L M \|s\|^2,
\]
for all \(s\), and~\eqref{eq:model-adequation-nonlipschitz} is satisfied.
% We provide generalizations and weaker requirements in \cref{sec:generalizations}.
% In \cref{sec:prox-qn}, we describe a method to compute a step satisfying~\eqref{eq:decrease_nonlipschitz_xi}.

In order to develop a convergence analysis, we further assume that the gradient of \(\varphi(\cdot; x_k)\) is Lipschitz continuous, which is satisfied, for instance, in the case of a quadratic model.
It is not necessary to assume at this point that those Lipschitz constants are uniformly bounded; we will make such an assumption when needed.
We gather the assumptions on the model from \cref{sec:tr-properties,sec:optim-measures} in \Cref{asm:parametric-all}.

\begin{modelassumption}%
  \label{asm:parametric-all}
  For any \(x \in \R^n\), \(\varphi(\cdot; x)\) is continuously differentiable with \(\varphi(0; x) = f(x)\) and \(\nabla_s \varphi(0; x) = \nabla f(x)\).
  In addition, \(\nabla_s \varphi(\cdot; x)\) is Lipschitz continuous with constant \(L(x)\) for all \(x \in \R^n\).
  Finally, \(\psi(\cdot; x)\) is proper, lsc, and satisfies \(\psi(0; x) = h(x)\) and \(\partial \psi(0; x) = \partial h(x)\).
\end{modelassumption}

The complete process is formalized in \Cref{alg:tr-nonsmooth},
which differs from a traditional trust-region algorithm in a few respects.
First, each iteration begins with the choice of a steplength \(\nu_k > 0\) for the proximal-gradient method.
Steplength \(\nu_k\) must be below \(1 / L(x_k)\) to ensure descent;
in addition, we connect \(\nu_k\) explicitly to \(\Delta_k\) for a reason that becomes apparent in \Cref{thm:deltasucc}.
% Because we assume that \(\psi(\cdot; x_k)\) is prox-bounded, we also implicitly assume throughout that \(\nu_k\) is chosen smaller than the threshold of prox-boundedness of \(\psi(\cdot; x_k)\).
Second, a step computation occurs in two phases. 
In the first phase, we compute the first step \(s_{k,1}\) of the proximal-gradient method applied to our model with trust-region radius \(\Delta_k\).
Step \(s_{k,1}\)  is an analog of the scaled projected gradient for nonsmooth regularized problems.
In the second phase, we continue the proximal-gradient iterations from \(s_{k,1}\) but possibly modify the trust-region radius so it does not exceed \(\beta \|s_{k,1}\|\) for a prescribed \(\beta \geq 1\).
This choice is similar in spirit to the analysis of \citet{curtis-lubberts-robinson-2018} for smooth problems, who set the radius to be proportional to the gradient norm.
More precisely, if \(\|s_{k,1}\| < \Delta_k\), we explore a trust region of radius \(\beta \|s_{k,1}\| \geq \|s_{k,1}\|\).
Because the constraint \(\|s\| \leq \Delta_k\) is inactive at \(s_{k,1}\), the first step of the proximal gradient method computed in the updated trust region remains \(s_{k,1}\), so that subsequent proximal gradient iterations will result in further decrease and the ultimate step \(s_k\) will satisfy~\eqref{eq:decrease_nonlipschitz_xi}.
If, on the other hand, \(\|s_{k,1}\| = \Delta_k\), the first step of the proximal gradient method computed in a larger trust region might differ from \(s_{k,1}\), which would jeopardize satisfaction of~\eqref{eq:decrease_nonlipschitz_xi}.
In order to preserve~\eqref{eq:decrease_nonlipschitz_xi}, we leave \(\Delta_k\) unchanged.

\begin{algorithm}
  \caption[caption]{%
    Nonsmooth Regularized Trust-Region Algorithm.%
    \label{alg:tr-nonsmooth}
  }
  \begin{algorithmic}[1]
    \State Choose constants
    \[
      0 < \eta_1 \leq \eta_2 < 1,
      \quad
      0 < \gamma_1 \leq \gamma_2 < 1 < \gamma_3 \leq \gamma_4
      \quad \text{and} \quad
      \alpha > 0, \, \beta \geq 1.
    \]
    \State Choose \(x_0 \in \R^n\) where \(h\) is finite, \(\Delta_0 > 0\), compute \(f(x_0) + h(x_0)\).
    \For{\(k = 0, 1, \dots\)}
      \State\label{alg:tr-nonsmooth:nuk}%
      Choose \(0 < \nu_k \leq  1 / (L(x_k) + \alpha^{-1} \Delta_k^{-1})\).
      \State Define \(m_k(s; x_k) := \varphi(s; x_k) + \psi(s; x_k)\) satisfying \Cref{asm:parametric-all}.
      \State Define \(m_k^{\nu}(s; x_k) := \varphi^{\nu}(s; x_k) + \psi(s; x_k)\) where \(\varphi^{\nu}(\cdot; x_k)\) is as in~\eqref{eq:varphi-pg}.
      \State\label{alg:tr-nonsmooth:sk1}%
      Compute \(s_{k,1}\) as the solution of~\eqref{eq:trsub} with model \(m_k^{\nu}(s; x_k)\).
      \State\label{alg:tr-nonsmooth:sk}%
      Compute an approximate solution \(s_k\) of~\eqref{eq:trsub} with model \(m_k(s; x_k)\) satisfying \Cref{asm:step_nonlip_xi} and such that \(\|s_k\| \leq \min(\Delta_k, \, \beta \|s_{k,1}\|)\).
      \State Compute the ratio
      \[
      \rho_k :=
      \frac{
        f(x_k) + h(x_k) - (f(x_k + s_k) + h(x_k + s_k))
      }{
        m_k(0; x_k) - m_k(s_k; x_k)
      }.
      \]
      \State If \(\rho_k \geq \eta_1\), set \(x_{k+1} = x_k + s_k\). Otherwise, set \(x_{k+1} = x_k\).
      \State Update the trust-region radius according to
      \[
        \Delta_{k+1} \in
        \left\{
          \begin{array}{lll}
             {[\gamma_3 \Delta_k, \, \gamma_4 \Delta_k]} &
             \text{ if } \rho_k \geq \eta_2, &
             \text{(very successful iteration)}
          \\ {[\gamma_2 \Delta_k, \, \Delta_k]} &
             \text{ if } \eta_1 \leq \rho_k < \eta_2, &
             \text{(successful iteration)}
          \\ {[\gamma_1 \Delta_k, \, \gamma_2 \Delta_k]} &
             \text{ if } \rho_k < \eta_1 &
             \text{(unsuccessful iteration).}
          \end{array}
        \right.
      \]
    \EndFor
  \end{algorithmic}
\end{algorithm}

\subsection{Convergence analysis and iteration complexity}%
\label{sec:tr_analysis}

Our first result states that a successful step is guaranteed provided the trust-region radius is small enough.

\begin{theorem}%
  \label{thm:deltasucc}
  Let \Cref{asm:parametric-all} and \Cref{asm:step_nonlip_xi} be satisfied and let
  \begin{equation}
    \label{eq:def-deltasucc}
    \Delta_{\textup{succ}} :=
    \frac{
      \kappa_{\textup{mdc}} (1 - \eta_2)
    }{
      2 \kappa_{\textup{m}} \alpha \beta^2
    } > 0.
  \end{equation}
  If \(x_k\) is not first-order stationary and
  \(
    \Delta_k \leq \Delta_{\textup{succ}}
  \),
  then iteration \(k\) is very successful and
  \(
    \Delta_{k+1} \geq \Delta_k
  \).
\end{theorem}

\begin{proof}
  Because \(x_k\) is not first-order stationary, \(s_{k,1} \neq 0\) and \(s_k \neq 0\).
  Note first that~\eqref{eq:pg1-decrease-palm},~\eqref{eq:optim-measure} and \Cref{asm:parametric-all} give
  \[
    \xi(\Delta_k; x_k, \nu_k) \geq
    (f + h)(x_k) - (\varphi + \psi)(s_1; x_k) \geq
    \tfrac{1}{2} (\nu_k^{-1} - L(x_k)) \|s_{k,1}\|^2.
  \]
  Line~\ref{alg:tr-nonsmooth:nuk} of \Cref{alg:tr-nonsmooth} implies in turn that \(\nu_k^{-1} - L(x_k) \geq \alpha^{-1} \Delta_k^{-1}\), so that
  \[
    \xi(\Delta_k; x_k, \nu_k) \geq
    \tfrac{1}{2} \alpha^{-1} \Delta_k^{-1} \|s_{k,1}\|^2.
  \]
  \Cref{asm:parametric-all} and \Cref{asm:step_nonlip_xi} together with the bound \(\|s_k\| \leq \beta \|s_{k,1}\|\) yield
  \begin{align*}
    |\rho_k - 1| & =
    \left| \frac{f(x_k + s_k) + h(x_k + s_k) - m_k(s_k; x_k)}{m_k(0; x_k) - m_k(s_k; x_k)} \right|
    \\ & \leq
    \frac{
      \kappa_{\textup{m}} \|s_k\|^2
    }{
      \kappa_{\textup{mdc}} \xi(\Delta_k; x_k, \nu_k)
    }
    \\ & \leq
    \frac{
      \kappa_{\textup{m}} \beta^2 \|s_{k,1}\|^2
    }{
      \tfrac{1}{2} \alpha^{-1} \Delta_k^{-1} \|s_{k,1}\|^2
    }
    \\ & =
    \frac{
      2 \kappa_{\textup{m}} \alpha \beta^2
    }{
      \kappa_{\textup{mdc}}
    } \Delta_k.
  \end{align*}
  Therefore, \(\Delta_k \leq \Delta_{\textup{succ}}\) implies \(\rho_k \geq \eta_2\) and iteration \(k\) is very successful.
  The trust-region update of \Cref{alg:tr-nonsmooth} ensures that \(\Delta_{k+1} \geq \Delta_k\).
\end{proof}

A careful examination of the proof of \Cref{thm:deltasucc} reveals that the model adequacy condition~\eqref{eq:model-adequation-nonlipschitz} could be replaced with the weaker condition
\begin{equation}
  \label{eq:model-adequation-nonlipschitz-sk1}
  |f(x_k + s_k) + h(x_k + s_k) - m_k(s_k; x_k)| \leq \kappa_{\textup{m}} \beta^2 \|s_{k,1}\|^2,
\end{equation}
which encapsulates the step size and the trust-region radius simultaneously, and suggests that \(s_{k,1}\) is the appropriate generalization of the projected gradient for nonsmooth regularized optimization.

We are now in position to show that \Cref{alg:tr-nonsmooth} identifies a first-order critical point.
We first consider the case where there are finitely many successful iterations.

\begin{theorem}%
  \label{thm:tr-finite}
  Let \Cref{asm:parametric-all} and \Cref{asm:step_nonlip_xi} be satisfied.
  If \Cref{alg:tr-nonsmooth} only generates finitely many successful iterations, then \(x_k = x^*\) for all sufficiently large \(k\) and \(x^*\) is first-order critical.
\end{theorem}

\begin{proof}
  The proof mirrors that of \citet[Theorem~\(6.4.4\)]{conn-gould-toint-2000}.
  Under the  assumptions given, there exists \(k_0 \in \N\) such that all iterations \(k \geq k_0\) are unsuccessful and \(x_k = x_{k_0} = x^*\).
  Assume by contradiction that \(x^*\)  is not first-order critical.
  The mechanism of \Cref{alg:tr-nonsmooth} ensures that \(\Delta_k\) decreases on unsuccessful iterations.
  Thus, there must be \(k_1 \geq k_0\) such that \(\Delta_k \leq \Delta_{\textup{succ}}\), where \(\Delta_{\textup{succ}}\) is defined in \Cref{thm:deltasucc}, which ensures that iteration \(k_1\) is successful and contradicts our assumption.
\end{proof}

We now turn to the case where there are infinitely many successful iterations and show that the objective is either unbounded below or a measure of criticality converges to zero.
The mechanism of \Cref{alg:tr-nonsmooth} and \Cref{thm:deltasucc} together ensure that
\begin{equation}
  \label{eq:deltamin}
  \Delta_k \geq \Delta_{\min}
  \quad \text{for all } k \in \N \text{ where} \quad
  \Delta_{\min} := \min(\Delta_0, \, \gamma_1 \Delta_{\text{succ}}) > 0.
\end{equation}
Thus, by definition of \(\xi(\cdot; x_k, \nu_k)\) and~\eqref{eq:deltamin}, we have
\begin{equation}
  \label{eq:xi-deltamin}
  \xi(\Delta_k; x_k, \nu_k) \geq \xi(\Delta_{\min}; x_k, \nu_k)
  \quad \text{for all } k \in \N.
\end{equation}
Following this last observation and in view of \Cref{prop:xi-critical} and~\eqref{eq:decrease-smooth-palm}, we define
\(
  \nu_k^{-1} \xi{(\Delta_{\min}; x_k, \nu_k)}^{\frac12}
\)
as our measure of criticality.
Observe the similarity between this measure and \(\|G_{\nu_k}(0)\|\) defined in~\eqref{eq:moreau_grad}.

% The following result parallels \citep[Lemma~\(2.3\)]{cartis-gould-toint-2011}.

% \begin{lemma}%
%   \label{lem:deltasucc-epsilon}
%   Let \Cref{asm:parametric-all} and \Cref{asm:step_nonlip_xi} be satisfied.
%   Let \(C := (1 - \eta_2) \kappa_{\textup{mdc}} / \kappa_{\textup{m}}\) and \(0 < \epsilon < \min(1, \, 1/C)\).
%   Then
%   \[
%     \xi(1; x_k) \geq \epsilon \quad \Longrightarrow \quad \Delta_k \geq \min(\Delta_0, \, \gamma_1 C \epsilon).
%   \]
% \end{lemma}

% \begin{proof}
%   By \Cref{thm:deltasucc}, if \(\xi(1; x_k) \geq \epsilon\), then \(C_k \geq C \epsilon\) and
%   \[
%     \Delta_k \leq \min(C \epsilon, \, \sqrt{C \epsilon}) = C \epsilon
%     \quad \Longrightarrow \quad
%     \Delta_{k+1} \geq \Delta_k.
%   \]
%   The smallest value that \(\Delta_k\) can assume is thus \(\gamma_1 C \epsilon\), and would occur at an unsuccessful iteration where \(\Delta_{k-1}\) has a value slightly larger than \(C \epsilon\), unless \(\Delta_0 < \gamma_1 C \epsilon\), in which case that smallest value is \(\Delta_0\).
% \end{proof}

% Note that without loss of generality, we may assume that \(\kappa_{\textup{m}} \geq 1\) in~\eqref{eq:model-adequation-nonlipschitz} so that \(0 < C < 1\) and the condition on \(\epsilon\) in \Cref{lem:deltasucc-epsilon} is simply that \(0 < \epsilon < 1\).

Our objective is to establish that \(\liminf \nu_k^{-1} \xi(\Delta_{\min}; x_k, \nu_k) = 0\) provided \(f + h\) is bounded below.
While doing so, we also establish a complexity result.

Let \(\epsilon > 0\) be a stopping tolerance set by the user.
We are interested in determining the smallest iteration number \(k(\epsilon)\) at which we achieve the first-order optimality condition
\begin{equation}
  \label{eq:tr-stop}
  \nu_k^{-1} \xi{(\Delta_{\min}; x_k, \nu_k)}^{\frac12} \leq \epsilon \quad (0 < \epsilon < 1).
\end{equation}
We denote
\begin{subequations}%
  \label{eq:S-U-sets}
  \begin{align}
       \mathcal{S} & := \{ k \in \N \mid \rho_k \geq \eta_1 \},
    \\ \mathcal{S}(\epsilon) & := \{ k \in \mathcal{S} \mid k < k(\epsilon) \},
    \\ \mathcal{U}(\epsilon) & := \{k \in \N \mid k \not \in \mathcal{S} \text{ and } k < k(\epsilon) \},
  \end{align}
\end{subequations}
respectively the set of all successful iterations, the set of successful iterations for which~\eqref{eq:tr-stop} has not yet been attained, and the set of unsuccessful iterations before~\eqref{eq:tr-stop} is first attained.

We make the following additional assumption on the model.
\begin{modelassumption}%
  \label{asm:lipschitz-bounded}
  In \Cref{asm:parametric-all}, there exists \(L > 0\) such that \(0 < L(x_k) \leq L\) for all \(k \in \N\).
  In addition, we select \(\nu_k\) at line~\ref{alg:tr-nonsmooth:nuk} of \Cref{alg:tr-nonsmooth} in a way that there exists \(\nu_{\min} > 0\) such that \(\nu_k \geq \nu_{\min}\) for all \(k \in \N\).
\end{modelassumption}
We stress that it is not necessary to know the value of or estimate \(L\); only to ensure that such a constant exists, which may be achieved either by controling the norm of quasi-Newton approximations \citep{lotfi-bonniot-orban-lodi-2020} or employing exact Hessians and substituting one for a bounded approximation when its norm is too large.
Finally, in view of~\eqref{eq:deltamin}, there exists \(\nu_{\min} > 0\) satisfying the assumption.
For instance, choosing \(\nu_k := 1 / (L(x_k) + \alpha^{-1} \Delta_k^{-1})\) at each iteration ensures that \(\nu_k \geq \nu_{\min} := 1 / (L + \alpha^{-1} \Delta_{\min}^{-1}) > 0\).

The following two results parallel the now-classic complexity analysis of \citet{cartis-gould-toint-2011} and references therein.

\begin{lemma}%
  \label{lem:cmplx-successful}
  Let \Cref{asm:parametric-all,asm:lipschitz-bounded,asm:step_nonlip_xi} be satisfied.
  Assume there are infinitely many successful iterations and that \(f(x_k) + h(x_k) \geq {(f + h)}_{\textup{low}}\) for all \(k \in \N\).
  Then, for all \(\epsilon \in (0, \, 1)\),
  \begin{equation}
    \label{eq:bound-Seps}
    |\mathcal{S}(\epsilon)| \leq
    \frac{
      (f + h) (x_0) - {(f + h)}_{\textup{low}}
    }{
      \eta_1 \kappa_{\textup{mdc}} \nu_{\min}^2 \epsilon^2
    } =
    O(\epsilon^{-2}).
  \end{equation}
\end{lemma}

\begin{proof}
  If \(k \in \mathcal{S}(\epsilon)\), \cref{asm:lipschitz-bounded,asm:step_nonlip_xi} and~\eqref{eq:xi-deltamin} imply
  \begin{align*}
    f(x_k) + h(x_k) - f(x_k + s_k) - h(x_k + s_k) & \geq \eta_1 (m_k(0; x_k) - m_k(s_k; x_k))
    \\ & \geq
    \eta_1 \kappa_{\textup{mdc}} \xi(\Delta_k; x_k, \nu_k)
    \\ & \geq
    \eta_1 \kappa_{\textup{mdc}} \xi(\Delta_{\min}; x_k, \nu_k)
    \\ & \geq
    \eta_1 \kappa_{\textup{mdc}} \nu_k^2 \epsilon^2
    \\ & \geq
    \eta_1 \kappa_{\textup{mdc}} \nu_{\min}^2 \epsilon^2.
  \end{align*}

  Because \(f + h\) is bounded below by \({(f + h)}_{\textup{low}}\), summing the above inequalities over all \(k \in \mathcal{S}(\epsilon)\) yields
  \begin{equation*}
    (f + h) (x_0) - {(f + h)}_{\textup{low}}
    \geq \sum_{k \in \mathcal{S}(\epsilon)} (f + h) (x_k) - (f + h) (x_{k+1})
    \geq |\mathcal{S}(\epsilon)| \eta_1 \kappa_{\textup{mdc}} \nu_{\min}^2 \epsilon^2,
  \end{equation*}
  which establishes~\eqref{eq:bound-Seps}.
\end{proof}

In order to derive a similar bound on the total number of iterations before~\eqref{eq:tr-stop} is first attained, we need to bound the number of unsuccessful iterations.

\begin{lemma}%
  \label{lem:cmplx-unsuccessful}
  Under the assumptions of \Cref{lem:cmplx-successful},
  \begin{equation}
    \label{eq:bound-Ueps}
    |\mathcal{U}(\epsilon)| \leq
    % \frac{|\log ( \Delta_{\min} / \Delta_0 )|}{\log(\gamma_2)} + |\mathcal{S}(\epsilon)| \frac{\log(\gamma_4)}{|\log(\gamma_2)|} =
    \log_{\gamma_2}(\Delta_{\min} / \Delta_0) + |\mathcal{S}(\epsilon)| |\log_{\gamma_2}(\gamma_4)| =
    O(\epsilon^{-2}).
  \end{equation}
\end{lemma}

\begin{proof}
  Each unsuccessful iteration reduces the trust-region radius by a factor at least \(\gamma_2\), while at each successful iteration, \(\Delta_{k+1} \leq \gamma_4 \Delta_k\).
  Thus if \(k(\epsilon) - 1\) is the iteration index just before~\eqref{eq:tr-stop} occurs for the first time,
  \begin{equation*}%
    % \label{eq:U-bound}
    \Delta_{\min} \leq
    \Delta_{k(\epsilon) - 1} \leq
    \Delta_0 \gamma_2^{|\mathcal{U}(\epsilon)|} \gamma_4^{|\mathcal{S}(\epsilon)|}.
  \end{equation*}
  Taking logarithms on both sides and remembering that \(0 < \gamma_2 < 1\) gives
  \[
    |\mathcal{U}(\epsilon)| \log(\gamma_2) + |\mathcal{S}(\epsilon)| \log(\gamma_4) \geq
    \log( \Delta_{\min} / \Delta_0 ),
  \]
  and establishes~\eqref{eq:bound-Ueps}.
\end{proof}

% In the special case where \(\Delta_0 = \Delta_{\min}\), \Cref{lem:cmplx-unsuccessful} establishes that \Cref{alg:tr-nonsmooth} only performs successful iterations.

Finally, the total number of iteration until~\eqref{eq:tr-stop} is attained is given in the next result, which simply combines \Cref{lem:cmplx-successful} and \Cref{lem:cmplx-unsuccessful}.

\begin{theorem}%
  \label{thm:complexity-bound}
  Under the assumptions of \Cref{lem:cmplx-successful},
  \begin{equation}
    \label{eq:tr-nonsmooth-complexity}
    |\mathcal{S}(\epsilon)| + |\mathcal{U}(\epsilon)| =
    O(\epsilon^{-2}).
  \end{equation}
\end{theorem}

We use the update \(\Delta_{k+1} \in [\gamma_3 \Delta_k, \, \gamma_4 \Delta_k]\) on very successful iterations but other possibilities exist.
For instance, it is common to set \(\Delta_{k+1} = \max(\gamma_3 \|s_k\|, \, \Delta_k)\) instead.
\Cref{lem:cmplx-unsuccessful} continues to hold because on successful iterations, \(\Delta_{k+1} \leq \max(\gamma_3 \Delta_k, \, \Delta_k) = \gamma_3 \Delta_k\).

\citet{curtis-lubberts-robinson-2018} establish a complexity bound of \(O(\epsilon^{-2})\) by making \(\Delta_k\) proportional to an optimality measure---in their context of smooth optimization, they choose the gradient norm.
\citet{grapiglia-yuan-yuan-2016} study the convergence and complexity of a generic algorithm that has trust-region methods as a special case and obtain the \(O(\epsilon^{-2})\) complexity bound under stronger smoothness assumptions than ours.
Among others, they establish a bound for regularized optimization but also require \(h\) to be convex and globally Lipschitz continuous.
\citet{curtis-robinson-samadi-2017} describe a nonstandard trust-region algorithm with a stronger \(O(\epsilon^{-3/2})\) complexity bound.

A straightforward consequence of \Cref{thm:complexity-bound} is that if \(f + h\) is bounded below, a subsequence of the criticality measure converges to zero.

\begin{corollary}%
  \label{cor:conv-liminf}
  Let \Cref{asm:parametric-all,asm:lipschitz-bounded}, and \Cref{asm:step_nonlip_xi} be satisfied.
  If there are infinitely many successful iterations, then, either
  \[
    \lim_{k \to \infty} f(x_k) + h(x_k) \to -\infty
    \quad \text{or} \quad
    \liminf_{k \to \infty} \nu_k^{-1} \xi(\Delta_{\min}; x_k, \nu_k)^{\frac12} = 0.
  \]
\end{corollary}

\begin{proof}
  Follows directly from \Cref{thm:complexity-bound}.
\end{proof}

In order to give an interpretation of \Cref{cor:conv-liminf}, consider~\eqref{eq:tr-parametric} with \(\Delta = \Delta_{\min} > 0\) along with its value function \(p(\Delta_{\min}; x, \nu)\), optimal set \(P(\Delta_{\min}; x, \nu)\) and the optimality measure \(\xi(\Delta_{\min}; x, \nu)\), where \((x, \nu)\) now plays the role of the parameter.
Similar to \Cref{prop:tr-parametric}, though with slightly stronger assumptions than \Cref{asm:parametric}, we have the following result.

\begin{proposition}%
  \label{prop:p1}
  Let \Cref{asm:prob} be satisfied and consider~\eqref{eq:tr-parametric} with \(\varphi\) as in~\eqref{eq:varphi-pg}.
  Assume \(\psi\) is proper and lsc in the joint variables \((s, x)\) and \(\psi(s; x) + \chi(s; \Delta_{\min})\) is level-bounded in \(s\) locally uniformly in \(x\).
  Then, the domain of \(p(\Delta_{\min}; \cdot, \cdot)\) and \(P(\Delta_{\min}; \cdot, \cdot)\) is \(\R^n \times \{\nu \mid \nu > 0\}\).
  In addition,
  \begin{enumerate}
    \item\label{itm:p1-lsc} \(p(\Delta_{\min}; \cdot, \cdot)\) is proper continuous and for all \(x \in \R^n\) and \(\nu > 0\), \(P(\Delta_{\min}; x, \nu)\) is nonempty and compact.
    In addition,  \(\xi(\Delta_{\min}; \cdot, \cdot)\) is proper lsc;
    \item\label{itm:P1-osc} if \(\{x_k\} \to \bar{x}\) and \(\{\nu_k\} \to \bar{\nu} > 0\),
    % in such a way that \(p(\Delta_{\min}; x_k, \nu_k) \to p(\Delta_{\min}; \bar{x}, \bar{\nu})\),
    and for each \(k\), \(s_k \in P(\Delta_{\min}; x_k, \nu_k)\), then \(\{s_k\}\) is bounded and all its limit points are in \(P(\Delta_{\min}; \bar{x}, \bar{\nu})\).
    % \item\label{itm:P1-osc-sufficient}
    % if there exists \(\bar{s} \in P(\Delta_{\min}; \bar{x}, \bar{\nu})\) such that \(\psi(\bar{s}; \cdot)\) is continuous at \(\bar{x}\), then \(p(\Delta_{\min}; \cdot, \cdot)\) is continuous at \((\bar{x}, \bar{\nu})\), and \(\{p(\Delta_{\min}; x_k, \nu_k)\} \to p(\Delta_{\min}; \bar{x}, \bar{\nu})\) in part~\ref{itm:P1-osc}.
    % In addition, \(\xi(\Delta_{\min}; \cdot, \cdot)\) is lsc at \((\bar{x}, \bar{\nu})\).
  \end{enumerate}
\end{proposition}

\begin{proof}
  Because \(h\) is proper lsc,~\eqref{eq:optim-measure} implies that \(\xi(\Delta_{\min}; \cdot, \cdot)\) is proper whenever \(p(\Delta_{\min}; \cdot, \cdot)\) is proper and is lsc whenever \(p(\Delta_{\min}; \cdot, \cdot)\) is continuous.
  The latter holds because \(p(\Delta_{\min}; \cdot, \cdot)\) is the composition of \(\nabla f\), which is continuous, with the Moreau envelope of \(\psi(\cdot; x) + \chi(\cdot; \Delta)\), and such Moreau envelope is continuous in \((x, \nu)\)---see, \citep[Theorem~\(1.25\)]{rtrw}.
  The rest follows by \citep[Theorems~\(1.17\) and~\(7.41\)]{rtrw}.
\end{proof}

% Note that the convergence properties of \Cref{alg:tr-nonsmooth} are unchanged if on very successful iterations, we set \(\Delta_{k+1} \in [\gamma_3 \Delta_k, \, \gamma_4 \Delta_k]\) if \(\Delta_k \leq \Delta_{\max}\) and \(\Delta_{k+1} = \Delta_k\) otherwise, for some prescribed \(\Delta_{\max} > \Delta_0\).
% Such modification ensures that \(\Delta_k \leq \gamma_4 \Delta_{\max}\) for all \(k\).
% The only place in the analysis above touched by this change is~\eqref{eq:U-bound}, which remains valid because
% \[
%   \Delta_{\min} \leq
%   \Delta_{k(\epsilon) - 1} \leq
%   \Delta_0 \gamma_2^{|\mathcal{U}(\epsilon)|} \gamma_4^{|\mathcal{L}(\epsilon)|} \leq
%   \Delta_0 \gamma_2^{|\mathcal{U}(\epsilon)|} \gamma_4^{|\mathcal{S}(\epsilon)|},
% \]
% where \(\mathcal{L}(\epsilon)\) is the subset of \(\mathcal{S}(\epsilon)\) for which \(\Delta_k \leq \Delta_{\max}\), and which thus satisfies \(|\mathcal{L}(\epsilon)| \leq |\mathcal{S}(\epsilon)|\).
% Therefore, \Cref{lem:cmplx-unsuccessful} continues to hold.
% However, our modification to \Cref{alg:tr-nonsmooth} now implies the upper bound \(\nu_k \leq 1 / (L(x_k) + \alpha^{-1} \Delta_k^{-1}) \leq \alpha \Delta_k \leq \nu_{\max} := \alpha \gamma_4 \Delta_{\max}\).

By \Cref{cor:conv-liminf}, if \(f + h\) is bounded below, there is an index set \(\mathcal{K}\) such that \({\{\nu_k^{-1} \xi(\Delta_{\min}; x_k, \nu_k)^{\frac12}\}}_{\mathcal{K}} \to 0\).
Assume that \({\{(x_k, \nu_k)\}}_{\mathcal{K}}\) possesses a limit point and, without loss of generality, that \({\{(x_k, \nu_k)\}}_{\mathcal{K}} \to (\bar{x}, \bar{\nu})\) with \(\bar{\nu} > 0\).
That implies that \({\{\xi(\Delta_{\min}; x_k, \nu_k)\}}_{\mathcal{K}} \to 0\) because for all sufficiently large \(k\),
\[
  \nu_k^{-1} {\xi(\Delta_{\min}; x_k, \nu_k)}^{\frac12} \geq
  \tfrac{1}{2} \bar{\nu}^{-1} {\xi(\Delta_{\min}; x_k, \nu_k)}^{\frac12} \geq
  0.
\]
Under the assumptions of \Cref{prop:p1}, \(\xi(\Delta_{\min}; \cdot, \cdot)\) is lsc, which means exactly that
\[
  0 = \liminf_{k \in \mathcal{K}} \xi(\Delta_{\min}; x_k, \nu_k) = \xi(\Delta_{\min}; \bar{x}, \bar{\nu}),
\]
so that \(\bar{x}\) is first-order critical.

It turns out that a stronger conclusion holds without further assumptions; the following result implies that \emph{every} limit point of \(\{(x_k, \nu_k)\}\) determines a first-order critical point.
The proof follows the logic of \citep[Theorem 6.4.6]{conn-gould-toint-2000} but is significantly simpler due to the form of \Cref{asm:step_nonlip_xi} and~\eqref{eq:xi-deltamin}.

\begin{theorem}%
  \label{thm:conv-lim}
  Let \Cref{asm:parametric-all,asm:lipschitz-bounded,asm:step_nonlip_xi} be satisfied.
  If there are infinitely many successful iterations, %then, either
  \[
    \lim_{k \to \infty} f(x_k) + h(x_k) \to -\infty
    \quad \text{or} \quad
    \lim_{k \to \infty} \nu_k^{-1} \xi(\Delta_{\min}; x_k, \nu_k)^{\frac12} = 0.
  \]
\end{theorem}

\begin{proof}
  If \(\{\nu_k^{-1} \xi(\Delta_{\min}; x_k, \nu_k)^{\frac12}\} \not \to 0\), there exist \(\epsilon > 0\) and an infinite set \(\mathcal{K} \subset \mathcal{S}\) such that \(\nu_k^{-1} \xi(\Delta_{\min}; x_k, \nu_k)^{\frac12} \geq \epsilon\) for all \(k \in \mathcal{K}\).
  Because each \(k \in \mathcal{K}\) is a successful iteration, \Cref{asm:step_nonlip_xi} and~\eqref{eq:xi-deltamin} yield
  \begin{align*}
    (f + h)(x_k) - (f + h)(x_{k+1})
       & \geq \eta_1 \kappa_{\text{mdc}} \xi(\Delta_k; x_k, \nu_k)
    \\ & \geq \eta_1 \kappa_{\text{mdc}} \xi(\Delta_{\min}; x_k, \nu_k)
    \\ & \geq \eta_1 \kappa_{\text{mdc}} \nu_{\min}^2 \epsilon^2
  \end{align*}
  for all \(k \in \mathcal{K}\), which is a contradiction if \(\{f(x_k) + h(x_k)\}\) is not bounded below.
\end{proof}

  \section{Proximal-quasi-Newton trust-region method}%
\label{sec:prox-qn}

In this section, we consider the computation of a trust-region step and develop a special case of \Cref{prop:pg-descent} in which
\begin{equation}
  \label{eq:varphi-qn}
  \varphi(s; x) := f(x) + \nabla {f(x)}^T s + \tfrac{1}{2} s^T B s,
\end{equation}
where \(B = B^T\).
We assume that \(\Delta > 0\) is fixed.
For conciseness, we use the notation \(\varphi(s) := \varphi(s; x)\) and \(\psi(s) := \psi(s; x) + \chi(s; \Delta)\). 
% \begin{equation}
%   \label{eq:def-phi}
%   \phi(s) := \varphi(s) + \psi(s) + \chi(s).
% \end{equation}
We work under \Cref{asm:parametric-all}, i.e., we assume that \(\psi\) is proper and lsc with prox-boundedness coming from \(\chi(\cdot; \Delta)\).

\subsection{Computing a trust-region step}

The following result states a fundamental relationship between \(G_\nu\) and \(\partial \psi\).
\begin{lemma}%
  \label{lem:subdiffs}
    \item  Let \(s_{j+1}\)be given by~\eqref{eq:pg} and \(G_{\nu}(s_j)\)be defined by~\eqref{eq:moreau_grad}.
    Then,
    \begin{subequations}
      \begin{align}
        \label{eq:Gsubgrad}
        G_{\nu}(s_j) - \nabla \varphi(s_j) &\in \partial\psi(s_{j+1}).\\
        \label{eq:sudifferential_phi}
        (B - \nu^{-1} I) (s_{j+1} - s_j) &\in
        \nabla \varphi(s_{j+1}) + \partial\psi(s_{j+1}).
      \end{align}
    \end{subequations}
\end{lemma}
\begin{proof}%
  \eqref{eq:Gsubgrad} is a simple restatement of~\eqref{eq:prox-kkt} and~\eqref{eq:sudifferential_phi} results from adding \(\nabla \varphi(s_{j+1})\)to both sides of~\eqref{eq:prox-kkt} and substituting the gradient of \(\varphi\) using~\eqref{eq:varphi-qn}. 
\end{proof}

The next result shows that~\eqref{eq:pg} is a descent method when \(\varphi\) is a quadratic.

\begin{lemma}%[Model descent]
  \label{lem:step_bound}
  Let \(\{s_j\}\) be generated according to~\eqref{eq:pg}.
  For all \(j \geq 0\),
  \begin{subequations}
    \begin{align}
      \label{eq:phi_step_update}
      \psi (s_{j+1}) + \nabla \varphi{(s_j)}^T(s_{j+1} - s_j) \leq
      \psi (s_j) - \tfrac{1}{2}\nu^{-1} \|s_{j+1} - s_j\|^2,
      \\
      \label{eq:phi_descent_rel}
      (\varphi+\psi)(s_{j+1}) \leq
      (\varphi+\psi)(s_j) + \tfrac12 {(s_{j+1} - s_j)}^T (B - \nu^{-1} I) (s_{j+1} - s_j).
    \end{align}
  \end{subequations}
\end{lemma}

\begin{proof}
  Because \(s_{j+1}\) solves~\eqref{eq:pg},
  \[
    \tfrac{1}{2}\nu^{-1} \|s_{j+1} - (s_j - \nu \nabla \varphi(s_j))\|^2 + \psi(s_{j+1}) \leq \tfrac{1}{2}\nu^{-1} \|\nu \nabla \varphi(s_j)\|^2 + \psi(s_j). 
  \]
  By expanding the squared norm in the left-hand-side of the above and cancelling the common term \(\| \nu \nabla \varphi(s_j)\|^2\), we obtain~\eqref{eq:phi_step_update}.
  Because \(\varphi\) is quadratic,
  \begin{equation*}
    \varphi(s_{j+1}) =
    \varphi(s_j) + \nabla \varphi{(s_j)}^T (s_{j+1} - s_j) + \tfrac12 {(s_{j+1} - s_j)}^T B (s_{j+1} - s_j).
  \end{equation*}
  We now add \(\psi(s_{j+1})\) to both sides and use~\eqref{eq:phi_step_update} and obtain
  \begin{align*}
    (\varphi+\psi)(s_{j+1}) & \leq
    \varphi(s_j) + \psi(s_j)
     - \tfrac{1}{2}\nu^{-1} \|s_{j+1}  -s_j\|^2 + \tfrac12 {(s_{j+1} - s_j)}^T B (s_{j+1} - s_j)
    \\ & =  (\varphi+\psi)(s_j) + \tfrac12 {(s_{j+1} - s_j)}^T (B - \nu^{-1} I) (s_{j+1} - s_j).
    \tag*{\qed}
  \end{align*}
\end{proof}

We now examine two choices of \(\nu > 0\) that result in two decrease behaviors.

\begin{corollary}%
  \label{cor:step_bound}
  Under the assumptions of \Cref{lem:step_bound}, assume \(0 < \nu \leq (1 - \theta) / \|B\|\) for some \(\theta  \in (0, \, 1)\), or simply that \(\nu > 0\) if \(B = 0\), in which case \(\theta = 1\).
  Then, % for all \(j \geq 0\),
  \begin{equation}
    \label{eq:phi_descent_step}
    (\varphi+\psi)(s_{j+1}) \leq
    (\varphi+\psi)(s_j) - \tfrac12 \theta \nu^{-1} \, \|s_j - s_{j+1}\|^2,
    \quad
    (j \geq 0).
  \end{equation}
\end{corollary}

\begin{proof}
  If \(B = 0\),~\eqref{eq:phi_descent_step} with \(\theta = 1\) follows directly from~\eqref{eq:phi_step_update}.
  If \(B \neq 0\), we have by assumption \((1 - \theta) \nu^{-1} \geq \|B\|\), so that \(\lambda_{\max}(B - \nu^{-1} I) \leq -\theta \nu^{-1} < 0\), and therefore,
  \[
    {(s_{j+1} - s_j)}^T (B - \nu^{-1} I) (s_{j+1} - s_j) \leq
    - \theta \nu^{-1} \, \|s_{j+1} - s_j\|^2,
  \]
  which combines with~\eqref{eq:phi_descent_rel} to complete the proof.
\end{proof}

\begin{corollary}%
  \label{cor:step_bound_G}
  Under the assumptions of \Cref{lem:step_bound}, assume \(B \neq 0\), let \(0 < \theta < 1 / (4 \|B\|)\) and \(\nu^{\min} \leq \nu \leq \nu^{\max}\), where
  \[
    \nu^{\min} := \frac{1 - \sqrt{1 - 4 \theta \|B\|}}{2 \|B\|},
    \qquad
    \nu^{\max} := \frac{1 + \sqrt{1 - 4 \theta \|B\|}}{2 \|B\|}.
  \]
  Then, for all \(j \geq 0\),
  \begin{equation}
    \label{eq:phi_descent_G}
    (\varphi+\psi)(s_{j+1})\leq
    (\varphi+\psi)(s_j) - \tfrac12 \theta \nu^{-2} \, \|s_j - s_{j+1}\|^2 = (\varphi+\psi)(s_j) - \tfrac12 \theta \|G_{\nu}(s_j)\|^2.
  \end{equation}
\end{corollary}

\begin{proof}
  Under our assumptions, the quadratic \(p(\nu) := \|B\| \nu^2 - \nu + \theta\) has the two positive real roots \(\nu^{\min}\) and \(\nu^{\max}\).
  Moreover, for all \(\nu \in [\nu^{\min}, \, \nu^{\max}]\), \(p(\nu) \leq 0\), which can also be written \(\|B\| - \nu^{-1} \leq -\theta \nu^{-2}\).
  Therefore, if \(\nu \in [\nu^{\min}, \, \nu^{\max}]\), then for all \(j\),
  \[
    {(s_{j+1} - s_j)}^T (B - \nu^{-1} I) (s_{j+1} - s_j) \leq
    - \theta \nu^{-2} \, \|s_{j+1} - s_j\|^2 =
    - \theta \|G_{\nu}(s_j)\|^2,
  \]
  which combines with~\eqref{eq:phi_descent_rel} to complete the proof.
\end{proof}

% \begin{subequations}%
%   \begin{align}
%   (\varphi + \psi)(s; x) &\leq (f + h)(x) - \tfrac{1}{2}\theta\nu^{-1} \|s_1\|^2\label{eq:pg1-decrease-quad}\\
%   (\varphi + \psi)(s; x) &\leq (f + h)(x) - \tfrac{1}{2}\theta \|s_1\|^2, \quad \nu\in[\nu^{\min}, \nu^{\max}]\label{eq:pg1-decrease-quadtheta}
%   \end{align}
% \end{subequations}

% In the smooth case, where \(h = \psi = 0\), \(s_1 = -\nu \nabla f(x)\) so that the decrease is
% \begin{subequations}
%   \begin{align}
%   \varphi(s; x) &\leq f(x) - \frac{1}{2} \nu\theta \|\nabla f(x)\|^2,\label{eq:decrease-smooth-quad}\\
%   \varphi(s; x) &\leq f(x) - \frac{1}{2} \nu^2\theta \|\nabla f(x)\|^2, \quad \nu \in [\nu^{\min}, \nu^{\max}]. \label{eq:decrease-smooth-quadtheta}
%   \end{align}
% \end{subequations}
Because \(s_0 = 0\) and \((\varphi+\psi)(s_0) = f(x) + h(x) < +\infty\), if \(\nu\) is chosen as in \Cref{cor:step_bound} or \Cref{cor:step_bound_G},~\eqref{eq:pg} generates iterates \(\{s_j\}\) such that \(\{(\varphi+\psi)(s_j)\}\) is monotonically decreasing and all its terms are finite.
Finiteness implies that \(\|s_j\| \leq \Delta\) for all \(j \geq 0\), i.e., all iterates lie in the trust region.
In particular, for any \(j\geq 1\),
\begin{equation}
  \label{eq:model_decrease}
  m_k(s_{j+1}; x_k) \leq m_k(s_j; x_k) \leq m_k(s_1; x_k) = m_k^\nu(s_1; x_k),
\end{equation}
where \(m_k^\nu(s_1; x_k) = \tfrac{1}{2}\nu^{-1}\|s_1 +\nu\nabla f(x_k)\|_2^2 + (\psi + \chi)(s_1)\) and hence \(s_j\) satisfies the sufficient decrease condition~\eqref{eq:decrease_nonlipschitz_xi}, and the final equality results from the fact that \(s_1\) is the same for any model of the form~\eqref{eq:varphi-qn}.

With regards to proximal gradient convergence, two situations may occur.
In the first,~\eqref{eq:pg} results in \(s_{j_0 + 1} = s_{j_0}\) for a smallest index \(j_0 > 0\).
In that case,~\eqref{eq:prox-kkt} yields
\[
  0 \in \partial (\varphi+\psi)(s_{j_0}),
\]
i.e., we have identified a stationary point of~\eqref{eq:trsub} in a finite number of iterations, while decreasing the value of \(m_k\) at each iteration.
Otherwise, \(s_{j+1} \neq s_j\) for all \(j \geq 0\), and the next result establishes sub-linear convergence of the proximal gradient method~\eqref{eq:pg}.

\begin{theorem}%[Proximal gradient descent for quadratic \(\bm{+}\) nonsmooth]
  \label{thm:prox_grad_unconstrained}
  Let \(\{s_j\}\) be generated according to~\eqref{eq:pg} with \(\nu\) as in \Cref{cor:step_bound}.
  Denote \({(\varphi + \psi)}_{\textup{low}} := \inf (\varphi+\psi) > -\infty\).
  Let \(v_{j+1}\) denote the left-hand side of~\eqref{eq:sudifferential_phi}.
  For any \(N \geq 1\),
  \[
    \min_{j=0, \dots, N-1} \|v_{j+1}\| \leq
    \sqrt{\frac{2}{N \theta} (\nu^{-1} - \lambda_{\min}(B)) \left( (\varphi+\psi)(s_0) - {(\varphi + \psi)}_{\textup{low}} \right)}.
  \]
\end{theorem}

\begin{proof}
  We rearrange~\eqref{eq:phi_descent_step} and sum from iteration \(j = 0\) to iteration \(j = N - 1\):
  \[
    \sum_{j=0}^{N-1} \|s_j - s_{j+1}\|^2 \leq
    \frac{2 \nu}{\theta} ( (\varphi+\psi)(s_0) - (\varphi+\psi)(s_N) )
    \leq \frac{2 \nu}{\theta} ( (\varphi+\psi)(s_0) - {(\varphi + \psi)}_{\textup{low}} ).
  \]
  For any positive sequence \(\{c_j\}\),
  \[
    \min_{0 \leq j \leq N-1} c_j = \sqrt{\min_{0 \leq j \leq N-1} c_j^2} \leq \sqrt{\frac{1}{N} \sum_{j=0}^{N-1} c_j^2}.
  \]
  Therefore,
  \[
    \min_{0 \leq j \leq N-1} \|s_j - s_{j+1}\|
    %    & = \sqrt{\min_{0 \leq j \leq N-1} \|s_j - s_{j+1}\|^2}
    % \\ & \leq \sqrt{\frac{1}{N} \sum_{j=0}^{N-1} \|s_j - s_{j+1}\|^2}
    % \\ &
    \leq \sqrt{\frac{2 \nu}{N \theta} \left( (\varphi+\psi)(s_0) - {(\varphi + \psi)}_{\textup{low}} \right)}.
  \]
  Because \(\|v_{j+1}\| \leq \|B - \nu^{-1} I\| \, \|s_j - s_{j+1}\| = (\nu^{-1} - \lambda_{\min}(B)) \, \|s_j - s_{j+1}\| \leq \nu^{-1} \|s_j - s_{j+1}\|\), we obtain the desired result.
\end{proof}

When solving~\eqref{eq:trsub}, a reasonable stopping condition would be \(\|v_{j+1}\| \leq \epsilon\) for a user-chosen tolerance \(\epsilon > 0\).
\Cref{thm:prox_grad_unconstrained} indicates that such stopping condition is attained after \(N(\epsilon)\) iterations, where
\[
  N(\epsilon) =
  \left\lceil
  \frac{2}{\epsilon^2 \theta} (\nu^{-1} - \lambda_{\min}(B)) \left( (\varphi+\psi)(s_0) - {(\varphi + \psi)}_{\textup{low}} \right)
  \right\rceil.
\]

A result similar to \Cref{thm:prox_grad_unconstrained} can be established under the step size rule of \Cref{cor:step_bound_G}, with nearly identical proof. 

\begin{theorem}%
  \label{thm:prox_grad_unconstrained_G}
  Let \(\{s_j\}\) be generated according to~\eqref{eq:pg} with \(\nu\) as in \Cref{cor:step_bound_G} with \(0 < \theta < 1 / (4 \|B\|)\).
  Assume \(\psi\), and therefore \(\varphi+\psi\), is bounded below and denote \({(\varphi + \psi)}_{\textup{low}}:= \inf (\varphi+\psi) > -\infty\).
  For any \(N \geq 1\),
  \[
    \min_{j=0, \dots, N-1} \|G_\nu(s_{j+1})\| \leq
    \sqrt{\frac{2}{N \theta} \left( (\varphi+\psi)(s_0) - {(\varphi + \psi)}_{\textup{low}} \right)}.
  \]
\end{theorem}

  \section{Proximal Operators for Trust-Region Subproblems}%
\label{sec:prox-operators}

In this section, we develop techniques for computing~\eqref{eq:pg} for use in Steps~\ref{alg:tr-nonsmooth:sk1} and~\ref{alg:tr-nonsmooth:sk} of \Cref{alg:tr-nonsmooth}.
Many standard proximal operators for both convex and nonconvex prox-bounded functions \(\psi\) have been worked out \citep{beckFO,combettes2011proximal}, and new examples for nonconvex problems continuously appear.
Well-known examples include the firm-thresholding penalty \citep{gao1997waveshrink}, the SCAD penalty \citep{fan2001variable}, MCP penalty \citep{zhang2010nearly}, lower \(C^2\) functions \citep{Hare2009}, any \(\ell_p^p\)-seminorm for \(0 < p < 1\) \citep[Appendix A]{zheng2018unified}, and other exotic operators, see e.g. \citep[Table~\(1\)]{zheng2020relax}.
We refer to such functions \(\psi\) as \emph{prox-friendly}.
However, \Cref{alg:tr-nonsmooth} requires evaluating proximal operators for modified functions that combine a shift and a summation with an indicator function.
By \Cref{asm:parametric-all}, our model \(\psi(s; x) \approx h(x + s)\) must coincide with \(h\) in value and subdifferential at \(s = 0\).
In particular, the choice \(\psi(s; x) = h(x + s)\) seems natural when \(h\) itself is prox-friendly.
Here we consider
\begin{equation}
  \label{eq:prox_form}
  \psi(s; x) := h(x + s) + \chi(s; \Delta \mathbb{B}_{p}),
\end{equation}
where \(h\) is prox-friendly, \(x\) is a shift,
% \(\Delta \mathbb{B}_{p}\) is the \(\ell_p\)-norm ball of radius \(\Delta > 0\), \(\chi(u; \Delta \mathbb{B}_{p})\) is the indicator of \(\Delta \mathbb{B}_{p}\),
and \(p \in \{1, 2, \infty\}\).
Below, we provide closed form solutions and/or efficient routines for~\eqref{eq:prox_form} with focus on the following cases:
\begin{enumerate}
  \item for an arbitrary separable prox-friendly \(h\), we evaluate \(\prox{\nu \psi(\cdot; x)}\) by leveraging \(\prox{\nu h}\), but we restrict our attention to \(p = \infty\).
  This allows us to consider~\eqref{eq:prox_form} with \(h(x) = \lambda \|x\|_1\) and \(h(x)  = \lambda \|x\|_0\);
  \item we consider \(h(x) = \lambda \|x\|_1\) in~\eqref{eq:prox_form} for \(p = 2\).
  % Our computations can be generalized to other convex \(h\).
\end{enumerate}

\subsection{\texorpdfstring{\(\bm{p = \infty}\), \(\bm{h}\) separable}{Infinity-norm ball, h separable}}

For the special case of \(\mathbb{B}_{\infty}\),~\eqref{eq:def-prox-mapping} and~\eqref{eq:prox_form} yield
\begin{equation}
  \label{eq:hardprox}
  \prox{\nu \psi}(q) := \argmin{s} \tfrac{1}{2} \nu^{-1} \|s-q\|^2 + h(x + s) + \chi(s; \Delta \mathbb{B}_{\infty}).
\end{equation}
If \(h\) is separable, i.e., \(h(x) = \sum_i h_i(x_i)\),~\eqref{eq:hardprox} decouples in each coordinate:
\[
 \prox{\nu\psi}{(q)}_i = \argmin{s_i} \tfrac{1}{2} \nu^{-1} {(s_i-q_i)}^2 +  h_i(x_i + s_i) + \chi(s_i; [-\Delta, \Delta]).
\]
Using the change of variable
\(
  v_i = x_i + s_i
\),
we may rewrite
\[
 \prox{\nu\psi}{(q)}_i =
 \argmin{v_i} \{
   \tfrac{1}{2} \nu^{-1} {(v_i- x_i -q_i)}^2 + h_i(v_i) + \chi(v_i; [x_i - \Delta, x_i + \Delta])
  \} - x_i.
\]
If \(h\) is convex, we may work backwards from the form of the solution.
For any \(p_i \in \prox{\nu \psi}{(q)}_i\), either
\begin{enumerate}
  \item \(|p_i| < \Delta\), in which case \( p_i \in \prox{\nu h_i}{(q + x)}_i - x_i\);
  \item otherwise, \(|p_i| = \Delta\) by construction, and
  \begin{align*}
    \prox{\nu\psi}{(q)}_i & =
    \argmin{v_i = x_i \pm \Delta} (\tfrac{1}{2} \nu^{-1} {(v_i - (x_i + q_i))}^2 + h_i(v_i)) - x_i
    \\ & =
    \argmin{s_i = \pm \Delta} \tfrac{1}{2} \nu^{-1} {(s_i - q_i)}^2 + h_i(x_i + s_i)
    \subseteq \{-\Delta, \, \Delta\}.
  \end{align*}
\end{enumerate}
In such cases, the definition of convexity implies that set of bound-constrained solutions includes the projection of the unconstrained solutions into the bounds.
Because the objective of~\eqref{eq:hardprox} is strictly convex, equality holds:
\[
  \prox{\nu\psi}{(q)}_i =
  \{ \proj{[x_i - \Delta, x_i + \Delta]}(\prox{\nu h_i} {(q+x)}_i) \} - x_i =
  \proj{[-\Delta, \Delta]}(\prox{\nu h_i} {(q+x)}_i - x_i),
\]

For example, let \(h(x) = \lambda \|x\|_1\).
Then,
% \begin{equation}
  % \label{eq:1normProx}
  \begin{align*}
    \prox{\nu\psi}{(q)}_i  & =
    \proj{[-\Delta, \Delta]}(\prox{\nu \lambda |\cdot|}{(q+x)}_i - x_i) =
    \proj{[-\Delta, \Delta]}\left(
    \begin{cases}
       q_i - \nu \lambda  & \phantom{|} x_i + q_i \phantom{|} > \nu \lambda \\
      -x_i & |x_i + q_i| \leq \nu \lambda \\
       q_i + \nu \lambda & \phantom{|} x_i + q_i \phantom{|} < -\nu \lambda
    \end{cases}
    \right) \\
    & =
    \proj{[-\Delta, \Delta]}\left(\proj{[q_i - \nu \lambda, q_i + \nu \lambda]}(-x_i)\right).
  \end{align*}
% \end{equation}

When \(h\) is nonconvex, there may be a greater variety of cases.
For instance, if \(h(x) = \lambda \|x\|_0\), a global solution of~\eqref{eq:hardprox} may be one of the bounds, or either of the unconstrained local minimizers \(q\) and \(-x\) if they lie inside the bounds.
A simple strategy consists in evaluating the objective of~\eqref{eq:hardprox} at those four points and choosing one with lowest objective value.

\subsection{\texorpdfstring{\(\bm{p = 2}\), \(\bm{h(x) = \lambda \|x\|_1}\)}{Euclidean-norm ball, L1-norm regularizer}}

When using other norms to define the trust region, additional computations are required.
For certain norms, we can dualize \(h\) to solve~\eqref{eq:hardprox}.
We focus on \(h(x) = \lambda \|x\|_1\) with an \(\ell_2\)-norm trust-region throughout because the \(\ell_2\)-norm is standard in the literature, and is used in \cref{sec:bpdn}.

First, we rewrite the scaled \(\ell_1\)-norm using its conjugate:
\[
  \lambda\|x + s\|_1 = \sup_{w\in \lambda \mathbb{B}_{\infty}} w^T (x + s),
\]
recharacterizing~\eqref{eq:def-prox-mapping} and~\eqref{eq:prox_form} as
\begin{equation}
  \label{eq:p2-norm1-dual}
  \min_s \sup_{w\in \lambda \mathbb{B}_{\infty}}
  \tfrac{1}{2} \nu^{-1}\|s-q\|^2 + w^T (x + s) + \chi(s; \Delta \mathbb{B}_2).
\end{equation}
Strong duality holds in this case since the objective is convex, piecewise linear-quadratic, and the primal solution is attained.
We interchange the order of minimization and maximization and complete squares in \(s\) and in \(w\) to obtain
\begin{equation}
\label{eq:exchanged}
\sup_{w\in \lambda \mathbb{B}_{\infty}}  \min_s \tfrac{1}{2} \nu^{-1}\|s -q+\nu w\|^2 + \chi(s; \Delta \mathbb{B}_2)  - \tfrac{1}{2} \nu^{-1}\left\|x +q - \nu w\right\|^2 + \tfrac{1}{2} \nu^{-1}\|x+q\|^2.
\end{equation}
The solution of the inner problem is
%Observe that the solution of the inner problem given the optimal dual variable \(w\) is the projection onto the \(\ell_2\)-norm ball
\begin{equation}
\label{eq:uw}
s(w) := \proj{\Delta \mathbb{B}_2}(q-\nu w).
\end{equation}
We substitute~\eqref{eq:uw} back into~\eqref{eq:exchanged} to rewrite the dual objective as
% This creates an explicit problem in \(w\) only; we can plug \(s(w)\) into the dual objective:
\begin{equation}
\label{eq:mydual}
  \sup_{w\in \lambda \mathbb{B}_{\infty}}
  \tfrac{1}{2} \nu^{-1} \dist{(q-\nu w; \Delta \mathbb{B}_2)}^2 - \tfrac{1}{2} \nu^{-1}\left\|x+q - \nu w\right\|^2 + \tfrac{1}{2} \nu^{-1}\|x+q\|^2.
\end{equation}
The change of variable
\begin{equation}
\label{eq:ychange}
y = q-\nu w,
\end{equation}
transforms~\eqref{eq:mydual} into
\begin{equation}
  \label{eq:hardprox_dual}
  \min_{q-\nu \lambda\mathbf{1} \leq y \leq q + \nu \lambda\mathbf{1}} \tfrac{1}{2} \nu^{-1} \left(\|y+x\|^2 - \dist{(y; \Delta \mathbb{B}_2)}^2\right),
\end{equation}
where \(\mathbf{1}\) is a vector of all ones.
As the value function of~\eqref{eq:p2-norm1-dual} with respect to \(s\), the objective of~\eqref{eq:hardprox_dual} is convex \citep[Proposition~\(2.22\)]{rtrw}.
The first-order optimality conditions of~\eqref{eq:hardprox_dual} are
\begin{equation}
  \label{eq:dual_condition}
  0 \in x + \frac{y}{\max\{1, \|y\|/\Delta \}}  + \nu\partial \chi(y; [q-\nu \lambda\mathbf{1}, q + \nu \lambda\mathbf{1}]).
\end{equation}
Once we have an optimal solution of~\eqref{eq:hardprox_dual} , denoted \(y^+\), we can evaluate~\eqref{eq:uw} at the corresponding \(w^+\) to obtain
\[
  s = \proj{\Delta \mathbb{B}_2}(y^+).
\]
which solves~\eqref{eq:hardprox}.
To characterize \(y^+\) more explicitly, we work backwards from properties of the solution.
There are only two possibilities to consider: \(y^+\) is in the trust region, and \(y^+\) is outside of the trust region.
\begin{enumerate}
  \item if \(\|y^+\| < \Delta\), \(\dist(y^+; \Delta \mathbb{B}_2) = 0\), and~\eqref{eq:hardprox_dual} and~\eqref{eq:dual_condition} simplify:
  \[
    % \label{eq:s_up_gen1}
    s = y^+ =\proj{[q-\nu \lambda\mathbf{1},  q + \nu \lambda\mathbf{1}]} (-x),
  \]
  where we used~\eqref{eq:uw} and~\eqref{eq:ychange};
  \item if \(\|y^+\| \geq \Delta\),~\eqref{eq:dual_condition} becomes
%   \begin{equation}
%     \label{eq:dual_cond}
%     0 \in x + \proj{\Delta \mathbb{B}_2}(y^+)  + \nu\partial \chi(y^+; [q-\nu \lambda\mathbf{1}, q + \nu \lambda\mathbf{1}]).
%   \end{equation}
% Using the closed form of the projection onto the 2-norm ball, we have
  \begin{equation*}
    % \label{eq:dual_cond_B2}
    0 \in x + \frac{\Delta}{\|y\|} y  + \nu\partial \chi(y; [q-\nu \lambda\mathbf{1}, q + \nu \lambda\mathbf{1}]).
  \end{equation*}
  Multiplying through by \(\|y\|/\Delta\) yields
\begin{equation}
\label{eq:inclusion}
    0 \in
    y + \frac{\|y\|}{\Delta} x  + \frac{\nu\|y\|}{\Delta} \partial \chi(y; [q-\nu \lambda\mathbf{1}, q + \nu \lambda\mathbf{1}]).
  \end{equation}
Suppose first that \(\eta := \|y^+\|\) is known.
A solution \(y^+\) to~\eqref{eq:inclusion} can be obtained by solving
\[
\min_{y \in [q-\nu \lambda\mathbf{1}, q + \nu \lambda\mathbf{1}]} \tfrac{1}{2}\|y + \frac{\eta}{\Delta} x\|^2
\]
which can be written in closed form as
  \begin{equation}
  \label{eq:MagicProj}
    y = \proj{[q-\nu \lambda\mathbf{1}, q + \nu \lambda\mathbf{1}]} \left(- \frac{\eta}{\Delta} x \right).
  \end{equation}
  Taking the norm of each side of~\eqref{eq:MagicProj} gives a scalar root finding equation that characterizes \(\eta\):
  %  The solution $y^+$ now requires knowing its norm in advance; we can solve this via a scalar root finding problem, where we seek \(\eta\) such that
  \[
    \eta = \left\|\proj{[q-\nu \lambda\mathbf{1}, q + \nu \lambda\mathbf{1}]} \left(-\frac{\eta}{\Delta}x \right)\right\|.
  \]
%  A good starting point for \(\eta\) is \(\eta = \Delta\).
Once we have solved for \(\eta = \|y^+\|\), we obtain \(y^+\) from~\eqref{eq:MagicProj}, and, using~\eqref{eq:uw},
  \[
    s = \proj{\Delta \mathbb{B}_2}\left(\proj{[q-\nu \lambda\mathbf{1}, q + \nu \lambda\mathbf{1}]} \left(-\frac{\eta}{\Delta}x \right)\right) = \left(\proj{[q-\nu \lambda\mathbf{1}, q + \nu \lambda\mathbf{1}]} \left(-\frac{\eta}{\Delta}x \right)\right)\frac{\Delta}{\eta}.
  \]
%  where again $s = \prox{\nu \psi}(q)$ in \eqref{eq:hardprox}.
\end{enumerate}

  \section{A quadratic regularization variant}%
\label{sec:quad-reg}

We now describe a variant of the trust-region algorithm of the previous sections inspired by the modified Gauss-Newton scheme proposed by \citet{nesterov-2007} in the context of nonlinear least-squares problems.
Here again, \citet{cartis-gould-toint-2011} establish a complexity of \(O(\epsilon^{-2})\) iterations to attain a near-optimality condition under the assumption that \(h\) is convex and globally Lipschitz continuous.
In the sequel, we obtain the same complexity bound under \Cref{asm:prob}.
The quadratic regularization method decribed below is closely related to the standard proximal gradient method with the exception that it employs an adaptive steplength.
It may be used as an alternative to a linesearch-based proximal gradient method such as those of \citet{li-lin-2015} and \citet{bot-csetnek-laszlo-2016}.

In the quadratic regularization method, we use the linear model
\begin{equation}%
  \label{eq:def-phi-quad-reg}
  \varphi(s; x) = f(x) + \nabla f{(x)}^T s \approx f(x + s)
\end{equation}
together with a model of \(\psi(s; x)\) that satisfies \Cref{asm:parametric-all}.
The first difference is that in the present setting, the Lipschitz constant of \(\nabla \varphi(\cdot; x)\) is \(L(x) = 0\) for all \(x \in \R^n\).
The second difference is that we must now assume that \(\psi(\cdot; x)\) is prox-bounded.
At \(x\), we define
\begin{subequations}%
  \label{eq:qrsub}
  \begin{align}
    p(\sigma; x) & := \minimize{s} \ m(s; x, \sigma), \\
    P(\sigma; x) & := \argmin{s} \ m(s; x, \sigma),
  \end{align}
\end{subequations}
where
\begin{equation}%
  \label{eq:qr-model}
  m(s; x, \sigma) := \varphi(s; x) + \psi(s; x) + \tfrac{1}{2} \sigma \|s\|^2,
\end{equation}
and \(\sigma > 0\) is a regularization parameter.
From \(x\), the method computes a step \(s \in P(\sigma; x)\).
As earlier, let us also define
\begin{equation}%
  \label{eq:def-xi-qr}
  \xi(\sigma; x) := f(x) + h(x) - p(\sigma; x) \geq 0.
\end{equation}

If we combine~\eqref{eq:def-phi-quad-reg} with~\eqref{eq:qr-model}, we may write
\begin{equation}%
  \label{eq:qr-model-pg}
  m(s; x, \sigma) =
  \tfrac{1}{2} \sigma \|s + \sigma^{-1} \nabla f(x)\|^2 + \psi(s; x) + f(x) - \tfrac{1}{2} \sigma^{-1} \|\nabla f(x)\|^2,
\end{equation}
where the last two terms are independent of \(s\).
In~\eqref{eq:qr-model-pg}, we recognize a model of the form~\eqref{eq:pg1}, so that minimizing~\eqref{eq:qr-model} amounts to performing a single step of the proximal gradient method with step size \(1 / \sigma\) and Lipschitz constant \(L = 0\).
The decrease guaranteed by the proximal gradient method is given by~\eqref{eq:pg1-decrease-palm}, i.e.,
\begin{equation}%
  \label{eq:qr-pg1-decrease-m}
  \xi(x; \sigma) =
  f(x) + h(x) - m(s; x, \sigma) \geq
  \tfrac{1}{2} \sigma \|s\|^2,
\end{equation}
so that
\begin{equation}%
  \label{eq:qr-pg1-decrease}
  f(x) + h(x) - (\varphi(s; x) + \psi(s; x)) \geq
  \sigma \|s\|^2.
\end{equation}
Because of~\eqref{eq:qr-pg1-decrease}, there is no need for a sufficient decrease assumption such as~\eqref{eq:decrease_nonlipschitz_xi} in the quadratic regularization method.

In view of~\eqref{eq:qr-model-pg}, \Cref{prop:prox} applies to~\eqref{eq:qrsub}.
In particular, \(p(\sigma; x)\) is continuous in \((\sigma, x)\), and \(P(\sigma; x)\) is nonempty and compact for all \(\sigma > 0\).

By \Cref{prop:rtrw}, for any \(\sigma > 0\), if \(s \in P(\sigma; x)\), then \(0 \in \nabla f(x) + \partial \psi(s; x) + \sigma s\).
Thus, we have the following optimality result.

\begin{lemma}%
  \label{lem:qr-stationary}
  Let \Cref{asm:parametric-all} be satisfied, \(\psi(\cdot; x)\) be prox-bounded, and let \(\sigma > 0\).
  Then \(\xi(\sigma; x) = 0 \Longleftrightarrow 0 \in P(\sigma; x) \Longrightarrow x\) is first-order stationary for~\eqref{eq:nlp}.
\end{lemma}

As in the trust-region context, we require that the difference between the model and the actual objective be bounded by a multiple of \(\|s_k\|^2\):
\begin{stepassumption}%
  \label{asm:step_qr}
  There exists \(\kappa_{\textup{m}} > 0\) such that for all \(k\),
  \begin{equation}
    \label{eq:model-adequation-qr}
    |f(x_k + s_k) + h(x_k + s_k) - \varphi_k(s_k; x_k) - \psi(s_k; x_k)| \leq \kappa_{\textup{m}} \|s_k\|^2.
  \end{equation}
\end{stepassumption}

Once a step \(s\) has been computed, its quality is assessed by comparing the decrease in \(\varphi(\cdot; x) + \psi(\cdot; x)\) with that in the objective \(f + h\), similarly to \Cref{alg:tr-nonsmooth}.
If both are in strong agreement, \(\sigma\) decreases.
Otherwise, \(\sigma\) increases.
We state the overall algorithm as \Cref{alg:qr-nonsmooth}.

\begin{algorithm}
  \caption[caption]{%
    Nonsmooth quadratic regularization algorithm.%
    \label{alg:qr-nonsmooth}
  }
  \begin{algorithmic}[1]
    \State Choose constants \(0 < \eta_1 \leq \eta_2 < 1\) and \(0 < \gamma_3 \leq 1 < \gamma_1 \leq \gamma_2\).
    \State Choose \(x_0 \in \R^n\) where \(h\) is finite, \(\sigma_0 > 0\), compute \(f(x_0) + h(x_0)\).
    \For{\(k = 0, 1, \dots\)}
      \State Define \(m(s; x_k, \sigma_k)\) as in~\eqref{eq:qr-model} satisfying \Cref{asm:parametric-all} with \(L = 0\).
      \State Compute a solution \(s_k\) of~\eqref{eq:qrsub} such that \Cref{asm:step_qr} holds.
      \State Compute the ratio
      \[
      \rho_k :=
      \frac{
        f(x_k) + h(x_k) - (f(x_k + s_k) + h(x_k + s_k))
      }{
        \varphi(0; x_k) + \psi(0; x_k) - (\varphi(s_k; x_k) + \psi(s_k; x_k))
      }.
      \]
      \State If \(\rho_k \geq \eta_1\), set \(x_{k+1} = x_k + s_k\). Otherwise, set \(x_{k+1} = x_k\).
      \State Update the regularization parameter according to
      \[
        \sigma_{k+1} \in
        \begin{cases}
          [\gamma_3 \sigma_k, \, \sigma_k] & \text{ if } \rho_k \geq \eta_2,
          \\ [\sigma_k, \, \gamma_1 \sigma_k] & \text{ if } \eta_1 \leq \rho_k < \eta_2,
          \\ [\gamma_1 \sigma_k, \, \gamma_2 \sigma_k] & \text{ if } \rho_k < \eta_1.
        \end{cases}
      \]
    \EndFor
  \end{algorithmic}
\end{algorithm}

We now combine~\eqref{eq:qr-pg1-decrease} with~\Cref{asm:step_qr} into the following result.

\begin{theorem}%
  \label{thm:sigmasucc}
  Let \Cref{asm:parametric-all} and \Cref{asm:step_qr} be satisfied, \(\psi(\cdot; x_k)\) be prox-bounded for each \(k \in \N\), and let
  \begin{equation}
    \label{eq:def-sigmasucc}
    \sigma_{\textup{succ}} :=
    \kappa_{\textup{m}} / (1 - \eta_2)
     > 0.
  \end{equation}
  If \(x_k\) is not first-order stationary and
  \(
    \sigma_k \geq \sigma_{\textup{succ}}
  \),
  then iteration \(k\) is very successful and
  \(
    \sigma_{k+1} \leq \sigma_k
  \).
\end{theorem}

\begin{proof}
  Let \(s_k\) be the step computed at iteration~\(k\) of \Cref{alg:qr-nonsmooth}.
  Because \(x_k\) is not first-order stationary, \(s_k \neq 0\).
  \Cref{asm:step_qr} and~\eqref{eq:qr-pg1-decrease} combine to yield
  \[
    |\rho_k - 1| =
    \frac{
      |f(x_k + s_k) + h(x_k + s_k) - (\varphi(s_k; x_k) + \psi(s_k; x_k))|
    }{
      \varphi(0; x_k) + \psi(0; x_k) - (\varphi(s_k; x_k) + \psi(s_k; x_k))
    } \leq
    \frac{\kappa_{\textup{m}} \|s_k\|^2}{\sigma_k \|s_k\|^2}.
  \]
  After simplifying by \(\|s_k\|^2\), we obtain \(\sigma_k \geq \sigma_{\text{succ}} \Longrightarrow \rho_k \geq \eta_2\).
\end{proof}

\Cref{thm:sigmasucc} ensures existence of a constant \(\sigma_{\max} > 0\) such that
\begin{equation}%
  \label{eq:sigmamax}
  \sigma_k \leq
  \sigma_{\max} := \min(\sigma_0, \gamma_2 \sigma_{\textup{succ}}) > 0
  \quad \text{for all } k \in \N.
\end{equation}

A result analogous to \Cref{thm:tr-finite} holds for \Cref{alg:qr-nonsmooth}.
We omit the proof, as it is nearly identical.

\begin{theorem}%
  \label{thm:qr-finite}
  Let \Cref{asm:parametric-all} and \Cref{asm:step_qr} be satisfied, and \(\psi(\cdot; x_k)\) be prox-bounded for each \(k \in \N\).
  If \Cref{alg:qr-nonsmooth} only generates finitely many successful iterations, \(x_k = x^*\) for sufficiently large \(k\) and \(x^*\) is first-order critical.
\end{theorem}

According to \Cref{prop:prox} part~\ref{prop:prox-monotonic}, and the identification \(\nu = \sigma^{-1}\), \(p(\sigma; x)\) increases as \(\sigma\) increases, so that \(\xi(\sigma; x)\) decreases as \(\sigma\) increases, and~\eqref{eq:sigmamax} yields
% Thus, it follows from~\eqref{eq:sigmamax} that
\begin{equation}%
  \label{eq:ximin-qr}
  \xi(\sigma_k; x_k) \geq \xi(\sigma_{\max}; x_k)
  \quad \text{for all } k \in \N.
\end{equation}
\Cref{lem:qr-stationary},~\eqref{eq:qr-pg1-decrease} and~\eqref{eq:ximin-qr} suggest using \(\xi{(\sigma_{\max}; x_k)}^{\frac12}\) as stationarity measure.

Let \(\epsilon > 0\) be a tolerance set by the user and consider the sets~\eqref{eq:S-U-sets}.
We are now in position to establish complexity results analogous to those obtained for \Cref{alg:tr-nonsmooth}.
% The proofs are nearly identical and are omitted.
The proof is nearly identical and is omitted.

% \begin{lemma}%
%   \label{lem:qr-cmplx-successful}
%   Let \Cref{asm:parametric-all} and \Cref{asm:step_qr} be satisfied, and \(\psi(\cdot; x_k)\) be prox-bounded for each \(k \in \N\).
%   Assume there are infinitely many successful iterations and that \(f(x_k) + h(x_k) \geq {(f + h)}_{\textup{low}}\) for all \(k \in \N\).
%   Then, for all \(\epsilon \in (0, \, 1)\),
%   \begin{equation}
%     \label{eq:qr-bound-Seps}
%     |\mathcal{S}(\epsilon)| =
%     O(\epsilon^{-2}).
%   \end{equation}
% \end{lemma}

% \begin{lemma}%
%   \label{lem:qr-cmplx-unsuccessful}
%   Under the assumptions of \Cref{lem:qr-cmplx-successful},
%   \begin{equation}
%     \label{eq:qr-bound-Ueps}
%     |\mathcal{U}(\epsilon)| =
%     O(1).
%   \end{equation}
% \end{lemma}

% \begin{theorem}%
%   \label{thm:qr-complexity-bound}
%   Under the assumptions of \Cref{lem:qr-cmplx-successful},
%   \begin{equation}
%     \label{eq:qr-nonsmooth-complexity}
%     |\mathcal{S}(\epsilon)| + |\mathcal{U}(\epsilon)| =
%     O(\epsilon^{-2}) + O(1) =
%     O(\epsilon^{-2}).
%   \end{equation}
% \end{theorem}

\begin{theorem}%
  \label{thm:qr-complexity-bound}
  Let \Cref{asm:parametric-all} and \Cref{asm:step_qr} be satisfied, and \(\psi(\cdot; x_k)\) be prox-bounded for each \(k \in \N\).
  Assume there are infinitely many successful iterations and that \(f(x_k) + h(x_k) \geq {(f + h)}_{\textup{low}}\) for all \(k \in \N\).
  Then, for all \(\epsilon \in (0, \, 1)\),
  \begin{equation}
    \label{eq:qr-nonsmooth-complexity}
    |\mathcal{S}(\epsilon)| =
    O(\epsilon^{-2}),
    \quad
    |\mathcal{U}(\epsilon)| =
    O(\epsilon^{-2}),
    \quad
    |\mathcal{S}(\epsilon)| + |\mathcal{U}(\epsilon)| =
    O(\epsilon^{-2}).
  \end{equation}
\end{theorem}
  \section{Implementation and numerical results}%
\label{sec:numerical}

\Cref{alg:tr-nonsmooth,alg:qr-nonsmooth} are implemented in Julia \citep{bezanson-edelman-karpinski-shah-2017} and are available at \https{github.com/UW-AMO/TRNC}, along with scripts to reproduce our experiments.
Our design allows the user to choose a method to compute a step, an important feature given the nonstandard \(\psi_k + \chi_k\) operator.

We compare the performance of \Cref{alg:tr-nonsmooth} (TR) to other proximal quasi-Newton routines: PANOC \citep{stella-themelis-sopasakis-patrinos-2017} and ZeroFPR \citep{themelis-stella-patrinos-2017}.
PANOC can be viewed as a proximal gradient descent scheme accelerated by limited-memory BFGS steps.
It performs proximal gradient iterations with a backtracking linesearch, and then \(20\) quasi-Newton steps computed using the proximal gradient method.
ZeroFPR is similar, but takes a fixed number of quasi-Newton steps between each proximal gradient step;
it defaults to proximal gradient descent if no progress is made during the inner quasi-Newton steps.
To compare, we count gradient evaluations as well as proximal operator evaluations, but in our example problems, proximal evaluations are far cheaper than gradients.

In the following experiments, we set \(\psi(s; x_k) := h(x_k + s)\).
Our stopping criteria for~\Cref{alg:tr-nonsmooth} is \(\xi(\Delta_k; x_k, \nu_k)^{1/2}\), which we use as a proxy for the first-order error measure \(\nu_k^{-1} \xi(\Delta_{\min}; x_k, \nu_k)^{1/2}\) defined in~\eqref{eq:optim-measure}.
We set \(\Delta_0 := 1.0\).
We compute trust-region steps using the proximal-gradient (PG) method with step length chosen as in \Cref{cor:step_bound}, denoted TR-PG in figures and tables.
The user could choose accelerated variants for the subproblem,
including our quadratic regularization procedure \Cref{alg:qr-nonsmooth} (R2), signified by TR-R2.
In our experiments, the latter performed similarly to the proximal gradient method, although it typically required fewer inner iterations.
We use proximal operators that include both \(\psi(\cdot; x_k)\) and the indicator of the trust region as described in \cref{sec:prox-operators}.
The criticality measure used in the inner PG iterations is the norm of the subgradient~\eqref{eq:sudifferential_phi}, while that used in the R2 inner iterations is \(\xi(\sigma_k; x_k)^{1/2}\), which is a proxy for \(\xi(\sigma_{\max}; x_k)^{1/2}\).
We set the inner tolerance to
\[
  \min (0.01, \, \xi(\Delta_k; x_k + s_{k,1}, \nu_k)^{\frac12}) \, \xi(\Delta_k; x_k + s_{k,1}, \nu_k),
\]
which is inspired from inexact Newton methods to encourage fast local convergence.
Note that \(\xi\) is computed with the first step \(s_{k,1}\) from Line~\ref{alg:tr-nonsmooth:sk1} of \cref{alg:tr-nonsmooth}.

We use automatic differentiation as implemented in the ForwardDiff package \citep{forwarddiff} to obtain \(\nabla f(x)\) and construct limited-memory quasi-Newton approximations by way of the LinearOperators package \citep{linearoperators}.
Below, we use LSR1 and LBFGS approximations with memory~\(5\) for the BPDN and ODE examples, respectively.
% In all of our experiments, Hessian approximations remained sufficiently well conditioned that we did not use the strategy of \cref{sec:hess-bounded}.

\subsection{LASSO/BPDN}%
\label{sec:bpdn}

The first set of experiments concerns LASSO/basis pursuit de-noise (BPDN) problems, which arise in statistical \citep{tibshirani1996regression} and compressed sensing \citep{donoho2006compressed} applications.
We seek to recover a sparse signal \(x_{\textup{true}} \in \R^n\) given observed noisy data \(b \in \R^m\).
\(x_{\textup{true}}\) is a sparse vector containing mostly zeros and 10 values of \(\pm 1\) where both the index of the nonzero entry and \(\pm\) are randomly generated.

\begin{figure}[h]
  \centering
  \subfloat[Signal: \(h=\lambda\|\cdot\|_1\), \(\Delta\mathbb{B}_2\)]{\label{fig:bpdn_b2_x}\includetikzgraphics[width=.49\linewidth]{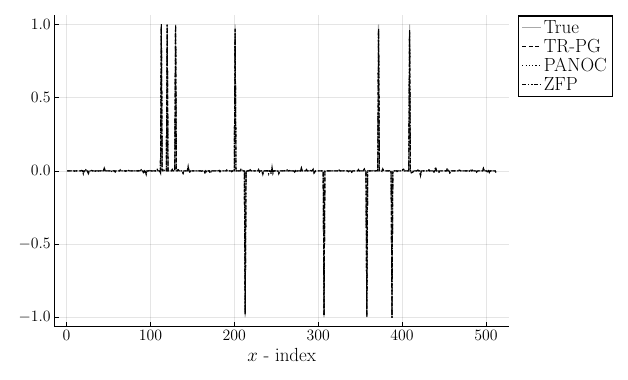}}
  \subfloat[History: \(h=\lambda\|\cdot\|_1\), \(\Delta\mathbb{B}_2\)]{\label{fig:bpdn_b2_obj}\includetikzgraphics[width=.49\linewidth]{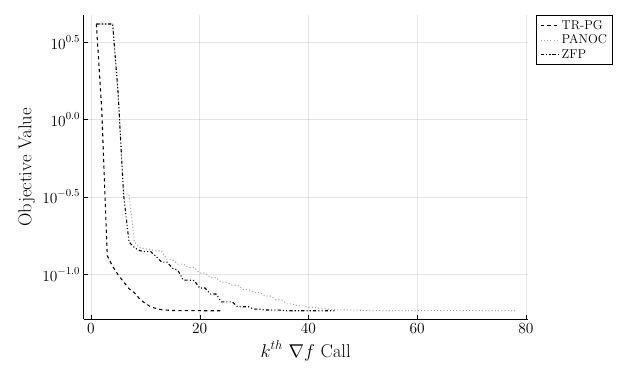}}
  \\
  \subfloat[Signal: \(h=\lambda\|\cdot\|_0\), \(\Delta\mathbb{B}_\infty\)]{\label{fig:bpdn_l0_x}\includetikzgraphics[width=.49\linewidth]{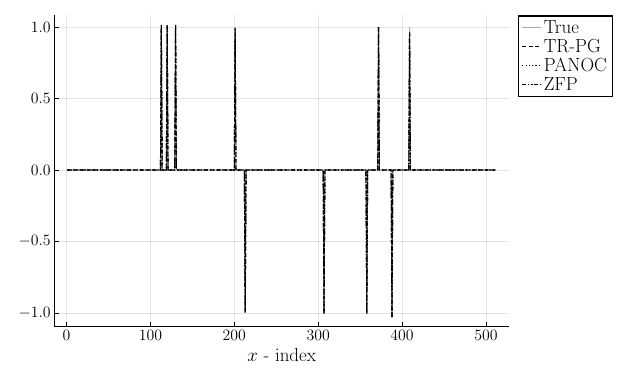}}
  \subfloat[History: \(h=\lambda\|\cdot\|_0\), \(\Delta\mathbb{B}_\infty\)]{\label{fig:bpdn_l0_obj}\includetikzgraphics[width=.49\linewidth]{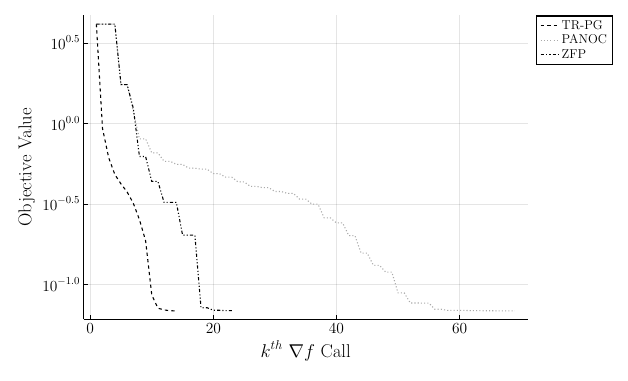}}
  \\
  \subfloat[Signal: \(h=\chi(\cdot; \lambda \mathbb{B}_0)\), \(\Delta\mathbb{B}_\infty\)]{\label{fig:bpdn_b0_x}\includetikzgraphics[width=.49\linewidth]{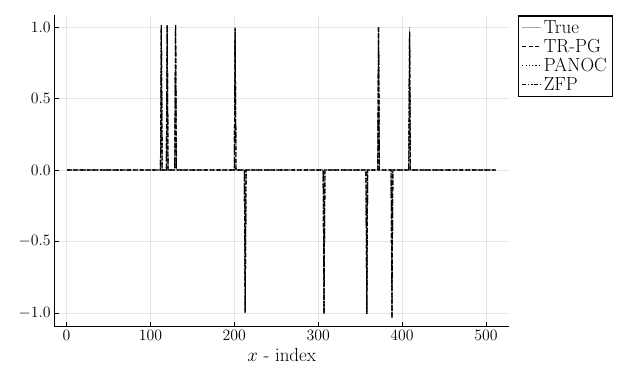}}
  \subfloat[History: \(h=\chi(\cdot; \lambda \mathbb{B}_0)\), \(\Delta\mathbb{B}_\infty\)]{\label{fig:bpdn_b0_obj}\includetikzgraphics[width=.49\linewidth]{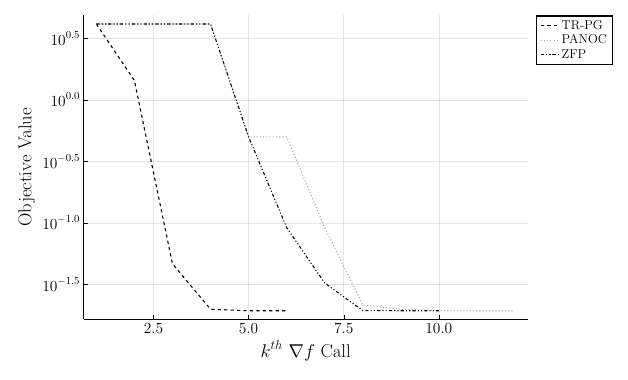}}
  \caption{%
    BPDN results~\eqref{eq:lasso} with \Cref{alg:tr-nonsmooth} and a proximal gradient subsolver (TR-PG), PANOC and ZeroFPR (ZFP): Signal plots (left) and objective value history (right).
    \(\Delta \mathbb{B}_p\) indicates the norm used to define the trust region.
  }%
  \label{fig:bpdn}
\end{figure}

We set \(m=200\), \(n = 512\), \(b := A x_{\textup{true}} + \varepsilon\) where \(\varepsilon \sim \mathcal{N}(0, .01)\) and \(A\) to have orthonormal rows---\(A^T\) is generated by taking the \(Q\) factor in the thin QR decomposition of a random \(n\)\(\times\)\(m\) matrix.
% Its singular values are either one or zero.
To recover \(x\), we solve
\begin{equation}
  \label{eq:lasso}
  \minimize{x} \tfrac{1}{2} \|Ax - b\|^2_2 + h(x).
\end{equation}
We first consider \(h(x) = \lambda \|x\|_p\) for \(p \in \{0, 1\}\) with \(\lambda = 0.1 \|A^T b\|_\infty\) in the vein of \citep{spgl1}, and employ both the \(\ell_2\) and \(\ell_\infty\) norms to define the trust region.
We also consider \(h(x) = \chi(x; \lambda\mathbb{B}_0)\) with \(\lambda = 10\) and an \(\ell_\infty\)-norm trust region.
We set the maximum number of inner iterations to \(5000\) and \(\epsilon = 10^{-3}\).
The quasi-Newton model is defined by a limited-memory SR1 approximation with memory \(5\).
All algorithms use \(x_0 = 0\).
% We set the scaling parameter to be  \(\lambda = .1\|A^T b\|_\infty\) in the vein of \citep{spgl1} for \(h=\ell_1,\ell_0\) and \(\lambda = 10\) (the number of sparse signals) for \(h = \lambda \mathbb{B}_0\).

\Cref{tab:bpdn} and \Cref{fig:bpdn} summarize our results.
% BPDN results for \(h = \ell_1, \ell_0, \lambda\mathbb{B}_0\) in~\eqref{eq:lasso} are shown in \Cref{tab:bpdn} and \Cref{fig:bpdn} with  \(\Delta\mathbb{B}_2\) and  \(\Delta\mathbb{B}_\infty\) trust regions, as well as \(\lambda \mathbb{B}_0\) with the \(\Delta\mathbb{B}_\infty\) trust region.
% The number of objective evaluations taken by each algorithm is denoted by \(N\).
\Cref{tab:bpdn} shows that \Cref{alg:tr-nonsmooth} performs comparably to PANOC and ZeroFPR in terms of parameter fit, it performs significantly fewer gradient evaluations and significantly more proximal operator evaluations.
% while comparing \Cref{tab:bpdn_b2} and \Cref{tab:bpdn_l0} shows that choice of trust-region does not affect the solution. % for this problem over choice of regularizer.
% Focusing on the number of function evaluations highlights the advantages of the TR approach over PANOC and ZeroFPR, which typically rely on more gradient evaluations at every iteration during linesearch and quasi-Newton steps.
Thus there is an advantage when proximal evaluations are cheap relative to gradient evaluations, especially in situations where the proximal operator of \(\psi\) is simpler or cheaper than that of \(h\).
% While the LASSO example effectively has the same complexity for objective and quadratic function solves, this difference grows large as the objective function becomes more complicated.
All algorithms yield nearly identical solution quality.
%perform well with nearly identical results.
The objective value history in \Cref{fig:bpdn} shows a steeper initial decrease for \Cref{alg:tr-nonsmooth} with shorter tails in all cases.
Results with R2 as subproblem solver are nearly identical though R2 performed fewer inner iterations than PG.

\begin{table}[hb]
  \captionsetup{position=top}
  \caption{%
    BPDN results~\eqref{eq:lasso} with \Cref{alg:tr-nonsmooth} and a proximal gradient subsolver (TR-PG), PANOC and ZeroFPR (ZFP).
    \(\Delta \mathbb{B}_p\) indicates the norm used in the trust region.
    The true value of \(h(\cdot)/\lambda\) is \(10\) for \(\|\cdot\|_1\) and \(\|\cdot\|_0\), but \(0\) for \(\chi(\cdot; \lambda \mathbb{B}_0)\).%
  }%
  \label{tab:bpdn}
  \begin{center}
    % \centering
    \footnotesize\setlength{\tabcolsep}{2pt}
\begin{tabular}{ cc |ccc|ccc|ccc}
   &    & \multicolumn{3}{|c|}{\(h=\lambda\|\cdot\|_1\), \(\Delta\mathbb{B}_2\)} & \multicolumn{3}{|c|}{\(h=\lambda\|\cdot\|_0\), \(\Delta\mathbb{B}_\infty\)} & \multicolumn{3}{|c}{\(h=\chi(\cdot; \lambda \mathbb{B}_0)\), \(\Delta\mathbb{B}_\infty\)} \\ \hline 
    & True & TR-PG & PANOC & ZFP & TR-PG & PANOC & ZFP & TR-PG & PANOC & ZFP \\
    \hline
    % Table body
    $ f(x) $ & 0.020 & 0.005 & 0.005 & 0.005 & 0.019 & 0.019 & 0.019 & 0.019 & 0.019 & 0.019 \\
    $ h(x)/\lambda $ & 10/0  & 10.750 & 10.767 & 10.750 & 10 & 10 & 10 & 0 & 0 & 0 \\
    $ \|x - x_{\text{true}}\|_2/\|A\| $ & 0 & 0.134 & 0.141 & 0.133 & 0.055 & 0.055 & 0.056 & 0.054 & 0.056 & 0.055 \\
    $ \nabla f $ evals &   & 24 & 78 & 45 & 14 & 69 & 23 & 6 & 12 & 10 \\
    $ \prox{\nu\psi}$ calls &   & 270 & 52 & 95 & 90 & 36 & 57 & 32 & 6 & 14 \\
  \hline
\end{tabular}

    % \subfloat[\(h=\|\cdot\|_1\), \(\Delta\mathbb{B}_2\).]{\input{figs/bpdn/LS_l1_B2/1/l1b2}}
    % % \label{tab:bpdn_b2}
    % \\
    % \subfloat[\(h=\|\cdot\|_0\), \(\Delta\mathbb{B}_\infty\).]{\input{figs/bpdn/LS_l0_Binf/1/l0binf}}
    % % \label{tab:bpdn_l0}
    % \\
    % \subfloat[\(h=\chi(\cdot; \lambda \mathbb{B}_0)\), \(\Delta\mathbb{B}_\infty\).]{\input{figs/bpdn/LS_B0_Binf/1/B0binf}}
    % % \label{tab:bpdn_B0}
  \end{center}
\end{table}

\subsection{A nonlinear inverse problem}

We next consider an inverse problem consisting in recovering the regularized solution to a system of nonlinear ODEs.
We seek parameters \(x_{\textup{true}} \in \R^n\) given observed noisy data \(b = F(x_{\textup{true}}) + \varepsilon\) where \(F: \R^{n} \to \R^m\) and \(\varepsilon \sim \mathcal{N}(0, 0.1)\).
The data generating mechanism \(F\) is given by the \citet{fitzhugh1955mathematical} and \citet{nagumo1962active} model for neuron activation
\begin{equation}%
  \label{eq:fitz}
  \frac{\mathrm{d}V}{\mathrm{d}t} = (V - V^3/3 - W + x_1)x_2^{-1},
  \quad
  \frac{\mathrm{d}W}{\mathrm{d}t} = x_2(x_3 V - x_4 W + x_5),
\end{equation}
which, if \(x_1 = x_4 = x_5 = 0\), becomes the \citet{van1926lxxxviii} oscillator
\begin{equation}%
  \label{eq:vdp}
  \frac{\mathrm{d}V}{\mathrm{d}t} = (V - V^3/3 - W)x_2^{-1},
  \quad
  \frac{\mathrm{d}W}{\mathrm{d}t} = x_2(x_3 V).
\end{equation}
Both models are highly nonlinear and ill-conditioned.

We use initial conditions \((V, W) = (2, 0)\) and discretize the time interval \([0, 20]\) at \(0.2\) second increments.
For given \(x\), let \(V(t; x)\) and \(W(t; x)\) be solutions of~\eqref{eq:fitz}.
Define variables \(v_i(x) \approx V(t_i; x)\), \(w_i(x) \approx W(t_i; x)\), \(i = 1, \dots, n+1\) where \(n = 20 / 0.2 = 100\).
We set \(F(x) := (v(x), w(x))\), where \(v(x) := (v_1(x), \dots, v_{n+1}(x))\) and \(w(x) := (w_1(x), \dots, w_{n+1}(x))\).
We generate \(b\) using \(x_{\textup{true}} = (0, 0.2, 1, 0, 0)\), which corresponds to a solve of the \citeauthor{van1926lxxxviii} oscillator.
To recover \(x\), we solve
\begin{equation}
  \label{eq:nonlin}
  \minimize{x} \tfrac{1}{2} \|F(x) - b\|^2_2 + h(x),
\end{equation}
with \(h(x) = \|x\|_0\).
ODE solves are performed with the DifferentialEquations.jl package \citep{rackauckas2017differentialequations}, which features an mechanism for choosing the solver, and provides \(\nabla v(x)\) and \(\nabla w(x)\) by way of automatic differentiation.
% Should the ODE solve fail to converge, \smarttodo{what do we do?}
% the condition number of the Hessian approximation is greater than  \(10^4\).
% We set the stopping tolerance to \(\epsilon = 10^{-3}\) in \Cref{alg:tr-nonsmooth}, PANOC and ZeroFPR\@.
We set \(\epsilon = 10^{-3}\) in all methods, the maximum iterations to \(500\), and use an LBFGS approximation of the Hessian.
For \Cref{alg:tr-nonsmooth}, the maximum number of inner iterations is \(5000\).
% , and the inner proximal gradient stopping tolerance to \(10^{-6}\).
%  and the maximum number of outer iterations to \(500\).

\Cref{tab:fhl0} summarizes our results and \Cref{fig:fhl0} shows overall data fit and objective function traces.
\Cref{alg:tr-nonsmooth} with either PG or R2 as subsolver, as well as ZeroFPR, correctly identified the nonzero pattern of \(x\) with reasonable error in the nonzero elements.
PANOC performs well initially, but its linesearch routine terminates prematurely as it generates a step length that is below a preset tolerance of \(10^{-7}\).
At that point, PANOC terminates.
ZeroFPR performs well, but needs many iterations to decrease the objective value to the same level as \Cref{alg:tr-nonsmooth}.
As in \cref{sec:bpdn}, \Cref{alg:tr-nonsmooth} converges with significantly fewer gradient evaluations than ZeroFPR, though with a significant number of proximal operator evaluations.
However, gradient evaluations in~\eqref{eq:fitz} are far more expensive and time consuming than proximal evaluations.
% From comparing parameter values, our algorithm can effectively enforce true sparsity along with ZeroFPR and PANOC in linear operator cases.
% That being said, in the process of creating these results, we discovered our algorithm is more robust to parameter and trust-region initialization.
\Cref{fig:fhl0} also reveals that the final iterate generated by \Cref{alg:tr-nonsmooth} and ZeroFPR results in trajectories that are visually indistinguishable from those associated with the exact solution.
\Cref{alg:tr-nonsmooth} with \Cref{alg:qr-nonsmooth} as a subsolver reaches a similar solution as ZeroFPR, but requires much fewer proximal and gradient evaluations.
The results appear in \Cref{tab:fhl0}.
Plots are nearly identical to those in \Cref{fig:fhl0}, and are hence omitted.

\begin{table}[h]
  \captionsetup{position=top}
  \caption{%
    Results for \Cref{alg:tr-nonsmooth} with proximal gradient (TR-PG) and \Cref{alg:qr-nonsmooth} (TR-R2) subsolvers, PANOC, and ZeroFPR applied to~\eqref{eq:fitz} with \(h = \|\cdot\|_0\), \(\Delta \mathbb{B}_\infty\) and LBFGS approximation.%
  }%
  \label{tab:fhl0}
  \begin{center}
  \footnotesize
 \setlength{\tabcolsep}{2pt}
 \begin{tabular}{  c  c  c  c  c   ||  c  c  c  c  c  c  }
    % Table header
\multicolumn{5}{c||}{Parameters} & \\ True & TR-PG & TR-R2 & PANOC & ZFP  & Measure & True & TR-PG & TR-R2 & PANOC & ZFP \\
    \hline
    % Table body
  0    & 0     & 0     & 0.840 & 0       & $ f(x) $                      & 1.058 & 1.078 & 1.266 & 73.888 & 1.048 \\
  0.2  & 0.170 & 0.130 & 0.690 & 0.188   & $ h(x) $                      & 2     & 2     & 3     & 5      & 3 \\
  1.0  & 1.136 & 1.408 & 0.952 & 1.048   & $ ||x - x_{\text{true}}||_2 $ & 0     & 0.139 & 0.427 & 1.636  & 0.051 \\
  0    & 0     & 0.107 & 0.983 & 0.010   & $ \nabla f $ evals            &       & 76    & 61    & 43     & 422 \\
  0    & 0     & 0     & 0.874 & 0       & $ \prox{\nu\psi}$ calls       &       & 60143 & 22617 & 30     & 421 \\
  \hline
\end{tabular}

  \end{center}
\end{table}

\begin{figure}[h]
  \centering
  \subfloat[Solution with data]{\label{fig:fh0_x}\includetikzgraphics[width=0.49\linewidth]{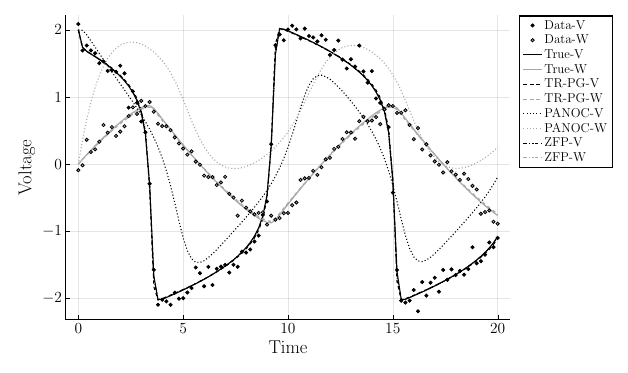}}
  \hfill
  \subfloat[Objective Function~\eqref{eq:nonlin} history]{\label{fig:fh0_obj}\includetikzgraphics[width=0.49\linewidth]{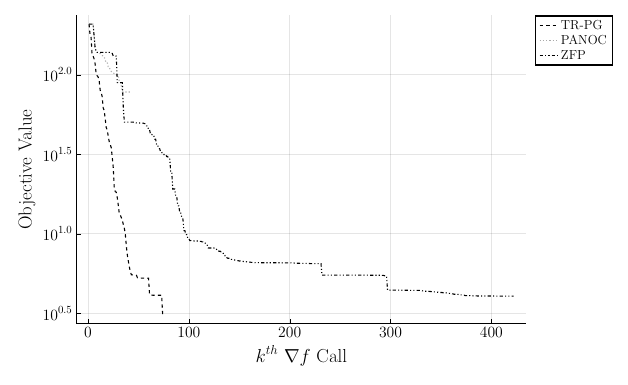}}
  % \subfloat[Inner iterations per trust-region iteration ]{\label{fig:fh0_comp}\includetikzgraphics[width=.5\linewidth]{figs/nonlin/FH/l0/complexity}}
  \caption{Solution of~\eqref{eq:fitz} with \(h(x) = \|x\|_0\) in~\eqref{eq:nonlin}, \(\Delta \mathbb{B}_\infty\) and LBFGS approximation.}%
  \label{fig:fhl0}
\end{figure}

We also compare \Cref{alg:qr-nonsmooth} to our own implementation of a standard proximal gradient with linesearch on~\eqref{eq:nonlin}.
We set the stopping tolerance for both to \(10^{-3}\). %, and a maximum of \(5000\) iterations.
\Cref{tab:fhl0pg} summarizes our results and \Cref{fig:fhl0pg} shows overall data fit and objective function traces.
Both \Cref{alg:qr-nonsmooth} and proximal gradient descent converge much slower than \Cref{alg:tr-nonsmooth}, where we use curvature information.
Neither algorithm correctly identified the nonzero pattern of \(x\) within \(5000\) iterations, although \Cref{alg:qr-nonsmooth} descends considerably faster than proximal gradient descent, and attains the stopping tolerance.  %\@; both exhibit slow convergence to the local minimum as \(\nabla f\) calls increase.
% Gradient evaluations in~\eqref{eq:fitz} are far more expensive and time consuming that proximal evaluations.
% From comparing parameter values, our algorithm can effectively enforce true sparsity along with ZeroFPR and PANOC in linear operator cases.
% That being said, in the process of creating these results, we discovered our algorithm is more robust to parameter and trust-region initialization.
\Cref{fig:fhl0pg} reveals that the final iterate generated by \Cref{alg:qr-nonsmooth} is closer to the solution than that of proximal gradient descent, though both terminated far from the correct answer.

\begin{table}[h]
  \captionsetup{position=top}
  \caption{%
    Results for \Cref{alg:qr-nonsmooth} (R2) and proximal gradient descent (PG) applied to~\eqref{eq:fitz} with \(h = \|\cdot\|_0\).%
  }%
  \label{tab:fhl0pg}
  \begin{center}
  \footnotesize
 \setlength{\tabcolsep}{3pt}
 \begin{tabular}{  c  c  c   ||  c  c  c  c  }
    % Table header
\multicolumn{3}{c||}{Parameters} & \\ True & R2 & PG  & Measure & True & R2 & PG \\
    \hline
    % Table body
  0 & 0 & 0.228   & $ f(x) $ & 1.058 & 3.852 & 24.246 \\
  0.200 & 0.142 & 0.245   & $ h(x) $ & 2 & 4 & 5 \\
  1.000 & 1.392 & 1.083   & $ ||x - x_{\text{true}}||_2 $ & 0 & 0.737 & 1.045 \\
  0 & 0.621 & 0.916   & $ \nabla f $ evals &   & 3892 & 5010 \\
  0 & 0.022 & 0.440   & $ \prox{\nu\psi}$ calls &   & 8891 & 5009 \\
  \hline
\end{tabular}

  \end{center}
\end{table}

\begin{figure}[h]
  \centering
  \subfloat[Solution with data]{\label{fig:fh0_x_pg}\includetikzgraphics[width=0.49\linewidth]{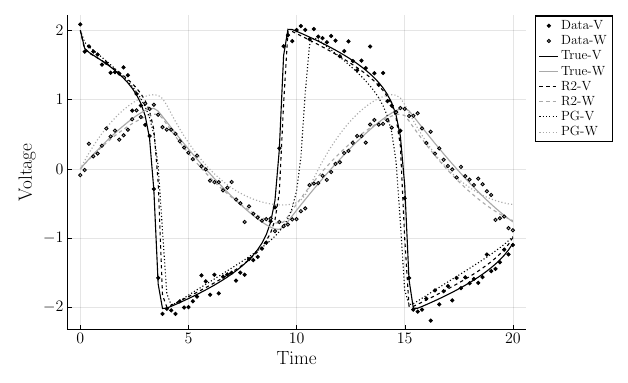}}
  \hfill
  \subfloat[Objective Function~\eqref{eq:nonlin} history]{\label{fig:fh0_obj_pg}\includetikzgraphics[width=0.49\linewidth]{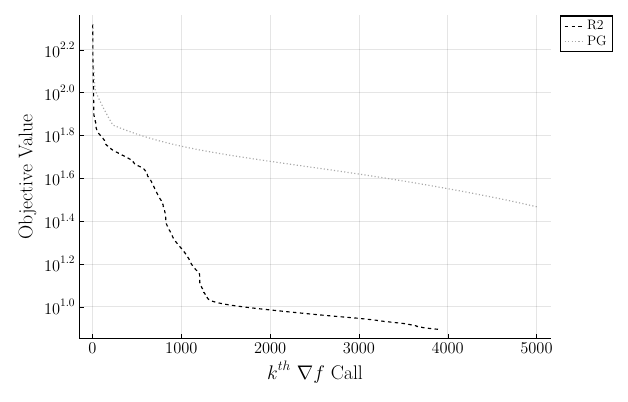}}
  % \subfloat[Inner iterations per trust-region iteration ]{\label{fig:fh0_comp}\includetikzgraphics[width=.5\linewidth]{figs/nonlin/FH/l0/complexity}}
  \caption{Solution of~\eqref{eq:fitz} for \(h = \|\cdot\|_0\) in~\eqref{eq:nonlin} with PG with linesearch and \Cref{alg:qr-nonsmooth}.}%
  \label{fig:fhl0pg}
\end{figure}

  \section{Discussion and perspectives}%
\label{sec:conclusion}

% One of the strengths of trust-region methods is to naturally allow nonconvex models to compute steps.
We demonstrated the performance of trust-region methods using quasi-Newton models against two linesearch methods constrained to LBFGS models, and observed faster convergence curves with fewer gradient evaluations.
Many regularizers in~\eqref{eq:nlp} have a closed-form or efficiently-computable proximal operator, whose cost is often dominated by that of a function or gradient evaluation in a large inverse problem.

The worst-case iteration complexity bound of \Cref{alg:tr-nonsmooth} matches the best known bound for trust-region methods in smooth optimization.
\Cref{alg:qr-nonsmooth}, a first-order method that is related to the proximal gradient method with adaptive steplength, does not require prior knowledge or estimation of a Lipschitz constant, and has a straightforward complexity analysis similar to that of \Cref{alg:tr-nonsmooth}. In practice, using curvature information in \Cref{alg:tr-nonsmooth} proved useful for efficiently estimating highly nonlinear nonsmooth models.
% We did not need the strategy of \cref{sec:hess-bounded} in our experiments.
Convergence of trust-region methods for smooth optimization can be established even if Hessian approximations are unbounded, provided they do not deteriorate too fast.
It may be possible to generalize our analysis along similar lines.

Interesting directions left to future work include implementation and analysis for \emph{inexact} function, gradient, and proximal operator evaluations, and extensions of our results to cubic regularization, and more general nonlinear stepsize control-type methods, such as those of \citep{grapiglia-yuan-yuan-2016}.

  \footnotesize
  \bibliographystyle{abbrvnat}
  \bibliography{abbrv,tr-nonsmooth}

  % \clearpage
  % \tableofcontents
  % \listoftodos\relax

\end{document}